\documentclass[11pt]{amsart}
\usepackage{amsxtra,amssymb,amsmath,amsthm,amsbsy}
\usepackage[all]{xy}

\theoremstyle{plain}

\numberwithin{equation}{section}

\newtheorem{theorem}{Theorem}[section]
\newtheorem{lemma}[theorem]{Lemma}
\newtheorem{definition-lemma}[theorem]{Definition-Lemma}
\newtheorem{proposition}[theorem]{Proposition}

\theoremstyle{definition}
\newtheorem{example}[theorem]{Example}
\newtheorem{remark}[theorem]{Remark}

\newenvironment{remarks}
{\vskip6pt \noindent {\bf Remarks:}} {\vskip6pt}

\newcommand{\id}         {{\mathrm {Id}}}
\newcommand{\Id}         {{\mathrm {Id}}}
\newcommand{\image}      {{\mathrm {im}}}
\newcommand{\Ker}        {{\mathrm {ker}}}

\newcommand{\pr}         {{\mathrm{pr}}}
\newcommand{\Diff}       {\mathrm{Diff}}

\newcommand{\Jac}        {\mathrm{Jac}}
\newcommand{\Cl}         {\mathrm{Cl}}

\newcommand{\SP} [1]     {{\left\langle {{#1}} \right\rangle}}

\newcommand{\frakg}     {\mathfrak{g}}

\newcommand{\gstar}     {\mathfrak{g}^*}

\newcommand{\T}         {\mathbb{T}}
\newcommand{\TN}        {\T N}
\newcommand{\TM}        {\T M}

\newcommand{\cC}       {\mathcal{C}}

\newcommand{\Cour}[1]      {[\![#1]\!]}

\newcommand{\Lie}        {\mathcal L}

\begin{document}
\title[]
{A brief introduction to Dirac manifolds}

\author[]{Henrique Bursztyn}

\address{Instituto de Matem\'atica Pura e Aplicada,
Estrada Dona Castorina 110, Rio de Janeiro, 22460-320, Brasil }
\email{henrique@impa.br}

\date{}

\maketitle

\begin{abstract}
These notes are based on a series of lectures given at the school on
\textit{Geometric and Topological Methods for Quantum Field Theory},
in Villa de Leyva, Colombia. We present a basic introduction to
Dirac manifolds, recalling the original context in which they were
defined, their main features, and briefly mentioning more recent
developments.
\end{abstract}

\tableofcontents

\section{Introduction}\label{sec:intro}
Phase spaces of classical mechanical systems are commonly modeled by
symplectic manifolds. It often happens that the dynamics governing
the system's evolution is constrained to particular submanifolds of
the phase space, e.g. level sets of conserved quantities (typically
associated with symmetries of the system, such as momentum maps), or
submanifolds resulting from constraints in the possible
configurations of the system, etc.
Any submanifold $C$ of a symplectic manifold $M$ inherits a
\textit{presymplectic} form (i.e., a closed 2-form, possibly
degenerate), given by the pullback of the ambient symplectic form to
$C$. It may be desirable to treat $C$ in its own right, which makes
presymplectic geometry the natural arena for the study of
constrained systems, see e.g. \cite{Got2,Gotay3}.

In many situations, however, phase spaces are modeled by more
general objects: Poisson manifolds (see e.g. \cite{MR}). A Poisson
structure on a manifold $M$ is a bivector field $\pi \in
\Gamma(\wedge^2 TM)$ such that the skew-symmetric bracket
$\{f,g\}:=\pi(df,dg)$ on $C^\infty(M)$ satisfies the Jacobi
identity. Just as for symplectic phase spaces, there are natural
examples of systems on Poisson phase spaces which are constrained to
submanifolds. The present notes address the following motivating
questions: what kind of geometric structure is inherited by a
submanifold $C$ of a Poisson manifold $M$? Can one ``pullback'' the
ambient Poisson structure on $M$ to $C$, similarly to what one does
when $M$ is symplectic? From another viewpoint, recall that $M$
carries a (possibly singular) symplectic foliation, which completely
characterizes its Poisson structure. Let us assume, for simplicity,
that the intersection of $C$ with each leaf $\mathcal{O}$ of $M$ is
a submanifold of $C$. Then $\mathcal{O}\cap C$ carries a
presymplectic form, given by the pullback of the symplectic form on
$\mathcal{O}$. So the Poisson structure on $M$ induces a
decomposition of $C$ into presymplectic leaves. Just as Poisson
structures define symplectic foliations, we can ask whether there is
a more general geometric object underlying foliations with
presymplectic leaves.

The questions posed in the previous paragraph naturally lead to
\textit{Dirac structures} \cite{courant,CouWe}, a notion that
encompasses presymplectic and Poisson structures. A key ingredient
in the definition of a Dirac structure on a manifold $M$ is the
so-called \textit{Courant bracket} \cite{courant} (see also
\cite{Dorf}), a bilinear operation on the space of sections of
$TM\oplus T^*M$ used to formulate a general integrability condition
unifying the requirements that a 2-form is closed and that a
bivector field is Poisson. These notes present the basics of Dirac
structures, including their main geometric features and key
examples. Most of the material presented here goes back to Courant's
original paper \cite{courant}, perhaps the only exception being the
discussion about morphisms in the category of Dirac manifolds in
Sec.~\ref{sec:morphisms}.

Despite its original motivation in constrained
mechanics\footnote{Dirac structures are named after Dirac's work on
the theory of constraints in classical mechanics (see e.g.
\cite{Dir,Sni}), which included a classification of constraint
surfaces (first class, second class...), the celebrated Dirac
bracket formula, as well as applications to quantization and field
theory.}, recent developments in the theory of Dirac structures are
related to a broad range of topics in mathematics and mathematical
physics. Due to space and time limitations, this paper is
\textit{not} intended as a comprehensive survey of this fast growing
subject (which justifies the omission of many worthy contributions
from the bibliography). A (biased) selection of recent aspects of
Dirac structures is briefly sketched at the end of the paper.

The paper is structured as follows. In Section \ref{sec:poisspre},
we recall the main geometric properties of presymplectic and Poisson
manifolds. Section~\ref{sec:dirac} presents the definition of Dirac
structures and their first examples. The main properties of Dirac
structures are presented in Section~\ref{sec:properties}.
Section~\ref{sec:morphisms} discusses morphisms between Dirac
manifolds. Section~\ref{sec:back} explains how Dirac structures are
inherited by submanifolds of Poisson manifolds.
Section~\ref{sec:survey} briefly mentions some more recent
developments and applications of Dirac structures.

\noindent {\bf Acknowledgments:} I would like to thank the
organizers and participants of the school in Villa de Leyva,
particularly Alexander Cardona for his invitation and encouragement
to have these lecture notes written up. Versions of these lectures
were delivered at Porto (\emph{Oporto Meeting on Geometry, Topology
and Physics, 2009}) and Canary Islands (\emph{Young Researchers
Workshop on Geometry, Mechanics and Control, 2010}) and helped me to
shape up the notes (which I hope to expand in the future); I thank
the organizers of these meetings as well. I am also indebted to P.
Balseiro, A. Cardona, L. Garcia-Naranjo, M. Jotz, C. Ortiz and M.
Zambon for their comments on this manuscript.

\subsection{Notation, conventions, terminology}\label{subsec:facts}
All manifolds, maps, vector bundles, etc. are smooth, i.e., in the
$C^\infty$ category. Given a smooth map $\varphi: M\to N$ and a
vector bundle $A\to N$, we denote the pullback of $A$ to $M$ by
$\varphi^* A \to M$.

For a vector bundle $E\to M$,  a \emph{distribution} $D$ in $E$
assigns to each $x\in M$ a vector subspace $D_x\subseteq E_x$. If
the dimension of $D_x$, called the \emph{rank} of $D$ at $x$, is
independent of $x$, we call the distribution \emph{regular}. A
distribution $D$ in $E$ is \emph{smooth} if for any $x\in M$ and
$v_0 \in D_x$, there is a smooth local section $v$ of $E$ (defined
on a neighborhood of $x$) such that $v(y)\in D_y$ and $v(x)=v_0$. A
distribution that is smooth and regular is a subbundle. The rank of
a smooth distribution is a lower semi-continuous function on $M$.
For a vector bundle map $\Phi: E\to A$ covering the identity, the
image $\Phi(E)$ is a smooth distribution of $A$; the kernel
$\ker(\Phi)$ is a distribution of $E$ whose rank is an upper
semi-continuous function, so it is smooth if and only if it has
locally constant rank. A smooth distribution $D$ in $TM$ is
\emph{integrable} if any $x\in M$ is contained in an
\textit{integral submanifold}, i.e., a connected immersed
submanifold $\mathcal{O}$ so that $D|_{\mathcal{O}}=T\mathcal{O}$.
An integrable distribution defines a decomposition of $M$ into
\textit{leaves} (which are the maximal integral submanifolds); we
generally refer to this decomposition of $M$ as a \textit{singular
foliation}, or simply a \textit{foliation}; see e.g.
\cite[Sec.~1.5]{DZ} for details. When $D$ is smooth and has constant
rank, the classical Frobenius theorem asserts that $D$ is integrable
if and only if it is involutive. We refer to the resulting foliation
in this case as \textit{regular}.

Throughout the paper, the Einstein summation convention is
consistently used.

\section{Presymplectic and Poisson structures}\label{sec:poisspre}

A symplectic structure on a manifold can be defined in two
equivalent ways: either by a nondegenerate closed 2-form or by a
nondegenerate Poisson bivector field. If one drops the nondegeneracy
assumption, the first viewpoint leads to the notion of a
\textit{presymplectic} structure, while the second leads to
\textit{Poisson} structures. These two types of ``degenerate''
symplectic structures  have distinct features that will be recalled
in this section.

\subsection{Two viewpoints to symplectic geometry}
Let $M$ be a smooth manifold. A 2-form $\omega \in \Omega^2(M)$ is
called \textit{symplectic} if it is nondegenerate and $d\omega=0$.
The nondegeneracy assumption means that the bundle map
\begin{equation}\label{eq:omegab}
\omega^\sharp: TM \to T^*M,\;\; X \mapsto i_X\omega,
\end{equation}
is an isomorphism; in local coordinates, writing $\omega =
\frac{1}{2}\omega_{ij} dx^i\wedge dx^j$, this amounts to the
pointwise invertibility of the matrix $(\omega_{ij})$. The pair
$(M,\omega)$, where $\omega$ is a symplectic 2-form, is called a
\textit{symplectic manifold}.

The basic ingredients of the hamiltonian formalism on a symplectic
manifold $(M,\omega)$ are as follows. For any function $f\in
C^\infty(M)$, there is an associated \textit{hamiltonian vector
field} $X_f \in \mathcal{X}(M)$, uniquely defined by the condition
\begin{equation}\label{eq:hamvf}
i_{X_f}\omega = df.
\end{equation}
In other words, $X_f = (\omega^\sharp)^{-1}(df)$. There is an
induced bilinear operation
$$
\{\cdot,\cdot\}: C^\infty(M)\times C^\infty(M)\to C^\infty(M),
$$
known as the \textit{Poisson bracket}, that measures the rate of
change of a function $g$ along the hamiltonian flow of a function
$f$,
\begin{equation}\label{eq:pbrk}
\{f,g\} := \omega(X_g,X_f)=\Lie_{X_f}g.
\end{equation}
The Poisson bracket is skew-symmetric, and one verifies from its
definition that
\begin{equation}\label{eq:ident}
d\omega(X_f,X_g,X_h) = \{f,\{g,h\}\} + \{h,\{f,g\}\} +
\{g,\{h,f\}\};
\end{equation}
it follows that the Poisson bracket satisfies the Jacobi identity,
since $\omega$ is closed. The pair $(C^\infty(M),\{\cdot,\cdot\})$
is a \textit{Poisson algebra}, i.e., $\{\cdot,\cdot\}$ is a Lie
bracket on $C^\infty(M)$ that is compatible with the associative
commutative product on $C^\infty(M)$ via the Leibniz rule:
$$
\{f, gh\} = \{f,g\} h + \{f,h\} g.
$$

It results from the Leibniz rule that the Poisson bracket is defined
by a bivector field $\pi \in \Gamma(\wedge^2TM)$, uniquely
determined by
\begin{equation}\label{eq:pi}
\pi(df,dg)=\{f,g\} = \omega(X_g,X_f);
\end{equation}
we write it locally as
\begin{equation}\label{eq:localp}
\pi = \frac{1}{2}\pi^{ij}\frac{\partial}{\partial{x^i}}\wedge
\frac{\partial}{\partial{x^j}}.
\end{equation}

The bivector field $\pi$ defines a bundle map
\begin{equation}\label{eq:pisharp}
\pi^\sharp: T^*M \to TM, \;\; \alpha\mapsto i_\alpha \pi,
\end{equation}
in such a way that $X_f = \pi^\sharp(df)$. Since $df =
\omega^\sharp(X_f) = \omega^\sharp(\pi^\sharp(df))$, we see that
$\omega$ and $\pi$ are related by
\begin{equation}\label{eq:equiv}
\omega^\sharp = (\pi^\sharp)^{-1}\;\;\;\mbox{ and }\;\;
(\omega_{ij}) = (\pi^{ij})^{-1}.
\end{equation}

The whole discussion so far can be turned around, in that one can
take the bivector field $\pi\in \Gamma(\wedge^2 TM)$, rather than
the 2-form $\omega$,  as the starting point to define a symplectic
structure. Given a bivector field $\pi \in \Gamma(\wedge^2TM)$, we
call it \textit{nondegenerate} if the bundle map \eqref{eq:pisharp}
is an isomorphism or, equivalently, if the local matrices
$(\pi^{ij})$ in \eqref{eq:localp} are invertible at each point. We
say that $\pi$ is \textit{Poisson} if the skew-symmetric bilinear
bracket $\{f,g\} = \pi(df,dg)$, $f,g \in C^\infty(M)$, satisfies the
Jacobi identity:
\begin{equation}\label{eq:jacobi}
\Jac_\pi(f,g,h):= \{f,\{g,h\}\} + \{h,\{f,g\}\} + \{g,\{h,f\}\} = 0,
\end{equation}
for all $f, g, h \in C^\infty(M)$.

The relation
$$
\pi(df,dg)=\omega(X_g,X_f)
$$
establishes a 1-1 correspondence between nondegenerate bivector
fields and nondegenerate 2-forms on $M$, in such a way that the
bivector field is Poisson if and only if the corresponding 2-form is
closed (see \eqref{eq:ident}). So a symplectic manifold can be
equivalently defined as a manifold $M$ equipped with a nondegenerate
bivector field $\pi$ that is Poisson.

The two alternative viewpoints to symplectic structures are
summarized in the following table:

\medskip

\begin{center}
\begin{tabular}{ |c | c| }
\hline
 nondegenerate $\pi \in \Gamma(\wedge^2 TM)$ & nondegenerate $\omega\in \Omega^2(M)$  \\ 
 $\Jac_\pi = 0$ & $d\omega = 0$   \\ 
 $X_f = \pi^\sharp(df)$ & $ i_{X_f}\omega = df$  \\ 
$\{f,g\} = \pi(df,dg)$ & $\{f,g\} = \omega(X_g,X_f)$\\ \hline
\end{tabular}
\end{center}

\medskip

Although the viewpoints are interchangeable, one may turn out to be
more convenient than the other in specific situations, as
illustrated next.

\subsection{Going degenerate}

There are natural geometrical constructions in symplectic geometry
that may spoil the nondegeneracy condition of the symplectic
structure, and hence take us out of the symplectic world. We mention
two examples.

Consider the problem of passing from a symplectic manifold $M$ to a
submanifold $\iota: C\hookrightarrow M$. To describe the geometry
that $C$ inherits from $M$, it is more natural to represent the
symplectic structure on $M$ by a 2-form $\omega$, which can then be
pulled back to $C$. The resulting 2-form $\iota^*\omega$ on $C$ is
always closed, but generally fails to be nondegenerate.

As a second example, suppose that a Lie group $G$ acts on a
symplectic manifold $M$ by symmetries, i.e., preserving the
symplectic structure, and consider the geometry inherited by the
quotient $M/G$ (we assume, for simplicity, that the action is free
and proper, so the orbit space $M/G$ is a smooth manifold). In this
case, it is more convenient to think of the symplectic structure on
$M$ as a Poisson bivector field $\pi$, which can then be projected,
or pushed forward, to $M/G$ since $\pi$ is assumed to be
$G$-invariant. The resulting bivector field on $M/G$ always
satisfies \eqref{eq:jacobi}, but generally fails to be
nondegenerate.

These two situations illustrate why one may be led to generalize the
notion of a symplectic structure by dropping the nondegeneracy
condition, and how there are two natural ways to do it. Each way
leads to a different kind of geometry: A manifold equipped with a
closed 2-form, possibly degenerate, is referred to as
\textit{presymplectic}, while a \textit{Poisson manifold} is a
manifold equipped with a Poisson bivector field, not necessarily
nondegenerate. The main features of presymplectic and Poisson
manifolds are summarized below.

\smallskip

\subsubsection{Presymplectic manifolds}\label{subsub:pre}\

On a presymplectic manifold $(M,\omega)$, there is a natural
\textit{null distribution} $K\subseteq TM$, defined at each point
$x\in M$ by the kernel of $\omega$:
$$
K_x := \Ker(\omega)_x = \{X\in T_xM\;|\; \omega(X,Y)=0 \; \forall\;
Y\in T_xM\}.
$$
This distribution is not necessarily regular or smooth. In fact, $K$
is a smooth distribution if and only if it has locally constant rank
(see Section~\ref{subsec:facts}). For $X,Y \in \Gamma(K)$, note that
$$
i_{[X,Y]}\omega = \Lie_X i_Y \omega - i_Y\Lie_X\omega = \Lie_X i_Y
\omega - i_Y(i_Xd + di_X)\omega =0;
$$
it follows that, when $K$ is regular, it is integrable by Frobenius'
theorem. We refer to the resulting regular foliation tangent to $K$
as the \textit{null foliation} of $M$.

One may still define hamiltonian vector fields on $(M,\omega)$ via
\eqref{eq:hamvf}, but, without the nondegeneracy assumption on
$\omega$, there might be functions admitting no hamiltonian vector
fields (e.g. if $df$ lies outside the image of \eqref{eq:omegab} at
some point). We say that a function $f\in C^\infty(M)$ is
\textit{admissible} if there exists a vector field $X_f$ such that
\eqref{eq:hamvf} holds. In this case,  $X_f$ is generally not
uniquely defined, as we may change it by the addition of any vector
field tangent to $K$. Still, the Poisson bracket formula
\begin{equation}\label{eq:pbrk1}
\{f,g\} = \Lie_{X_f} g
\end{equation}
is well defined (i.e., independent of the choice of $X_f$) when $f$
and $g$ are admissible. Hence the space of admissible functions,
denoted by
$$
C^\infty_{adm}(M) \subseteq C^\infty(M),
$$
is a Poisson algebra.

When $K$ is regular, a function is admissible if and only if
$df(K)=0$, i.e., $f$ is constant along the leaves of the null
foliation; in particular, depending on how complicated this
foliation is, there may be very few admissible functions (e.g., if
there is a dense leaf, only the constant functions are admissible).
When $K$ is regular and the associated null foliation is simple,
i.e., the leaf space $M/K$ is smooth and the quotient map $q: M\to
M/K$ is a submersion, then $M/K$ inherits a symplectic form
$\omega_{red}$, uniquely characterized by the property that
$q^*\omega_{red}=\omega$; in this case, the Poisson algebra of
admissible functions on $M$ is naturally identified with the Poisson
algebra of the symplectic manifold $(M/K, \omega_{red})$ via
$$
q^*: C^\infty(M/K)\stackrel{\sim}{\to} C^\infty_{adm}(M)
$$
(see e.g. \cite[Sec.~6.1]{OR} and references therein).


\subsubsection{Poisson manifolds}\label{subsub:poiss}\

If $(M,\pi)$ is a Poisson manifold, then any function $f\in
C^\infty(M)$ defines a (unique) hamiltonian vector field $X_f
=\pi^\sharp(df)$, and the whole algebra of smooth functions
$C^\infty(M)$ is a Poisson algebra with bracket
$\{f,g\}=\pi(df,dg)$.

The image of the bundle map $\pi^\sharp$ in \eqref{eq:pisharp}
defines a distribution on $M$,
\begin{equation}\label{eq:R}
R :=\pi^\sharp(T^*M)\subseteq TM,
\end{equation}
not necessarily regular, but always smooth and integrable. (The
integrability of the distribution $R$ may be seen as a consequence
of Weinstein's splitting theorem \cite{We83}.) So it determines a
singular foliation of $M$, in such a way that two points in $M$ lie
in the same leaf if and only if one is accessible from the other
through a composition of local hamiltonian flows. One may verify
that the bivector field $\pi$ is ``tangent to the leaves'', in the
sense that, if $f \in C^\infty(M)$ satisfies $\iota^*f \equiv 0$ for
a leaf $\iota: \mathcal{O}\hookrightarrow M$, then $X_f \circ \iota
\equiv 0$. So there is an induced Poisson bracket
$\{\cdot,\cdot\}_{\mathcal{O}}$ on $\mathcal{O}$ determined by
$$
\{f\circ \iota, g\circ \iota  \}_{\mathcal{O}} := \{f,g\}\circ
\iota,\;\;\; f,g \in C^\infty(M),
$$
which is nondegenerate; in particular, each leaf carries a
symplectic form, and one refers to this foliation as the
\textit{symplectic foliation} of $\pi$. The symplectic foliation of
a Poisson manifold uniquely characterizes the Poisson structure. For
more details and examples, see e.g. \cite{CW,DZ,MR}.

\begin{remark}\label{rem:liealg}
The integrability of the distribution \eqref{eq:R} may be also seen
as resulting from the existence of a Lie algebroid structure on
$T^*M$, with anchor $\pi^\sharp:T^*M\to TM$ and Lie bracket on
$\Gamma(T^*M)=\Omega^1(M)$ uniquely characterized by
$$
[df,dg]=d\{f,g\},
$$
see e.g. \cite{CW,CDW}; we will return to Lie algebroids in
Section~\ref{sec:properties}.
\end{remark}

\section{Dirac structures}\label{sec:dirac}

Dirac structures were introduced in \cite{courant,CouWe} as a way to
treat both types of ``degenerate'' symplectic structures, namely
 presymplectic and Poisson, in a unified manner. This common framework relies
on viewing presymplectic and Poisson structures as subbundles of
$$
\TM:= TM\oplus T^*M,
$$
defined by the graphs of the bundle maps
\eqref{eq:omegab} and \eqref{eq:pisharp}. The precise definition of
a Dirac structure resorts to additional geometrical structures
canonically present on $\TM$, as we now recall.

Let us consider $\mathbb{T}M$ equipped with the natural projections
$$\pr_{T}: \mathbb{T}M\to TM,\;\mbox{ and }\;
\pr_{T^*}:\mathbb{T}M\to T^*M,
$$
as well as two extra structures:  the nondegenerate, symmetric
fibrewise bilinear form $\SP{\cdot,\cdot}$ on $\TM$, given at each
$x\in M$ by
\begin{equation}\label{eq:pair}
\SP{(X,\alpha),(Y,\beta)}=\beta(X)+ \alpha(Y),\;\; X,Y \in T_xM,\;
\alpha,\beta \in T_x^*M,
\end{equation}
and the \textit{Courant bracket}
$\Cour{\cdot,\cdot}:\Gamma(\mathbb{T}M)\times \Gamma(\mathbb{T}M)
\to \Gamma(\mathbb{T}M)$,
\begin{equation}\label{eq:cour}
 \Cour{(X,\alpha),(Y,\beta)}=([X,Y],\Lie_X\beta - \Lie_Y\alpha
+\frac{1}{2}d(\alpha(Y)-\beta(X))).
\end{equation}

A \textit{Dirac structure} on $M$ is a vector subbundle $L\subset
\mathbb{T}M$ satisfying:
\begin{itemize}
\item[(i)] $L=L^\perp$, where the orthogonal is with respect to $\SP{\cdot,\cdot}$,
\item[(ii)] $\Cour{\Gamma(L),\Gamma(L)}\subseteq \Gamma(L)$, i.e., $L$ is involutive with respect
to $\Cour{\cdot,\cdot}$.
\end{itemize}

\begin{remarks}
\begin{itemize}
\item Since the pairing $\SP{\cdot,\cdot}$ has split signature,
condition (i) is equivalent to $\SP{\cdot,\cdot}|_L = 0$ and
$\mathrm{rank}(L)=\mathrm{dim}(M)$.

\item The Courant bracket satisfies
\begin{equation}\label{eq:jacobifails}
\Cour{\Cour{a_1,a_2}, a_3} + c.p. =
\frac{1}{3}d(\SP{\Cour{a_1,a_2},a_3} + c.p.),
\end{equation}
for $a_1, a_2, a_3 \in \Gamma(\mathbb{T}M)$, where $c.p.$ stands for
``cyclic permutations''. So it fails to satisfy the Jacobi identity,
and it is not a Lie bracket\footnote{The properties of the Courant
bracket are axiomatized in the notion of a \textit{Courant
algebroid} \cite{LWX}, see Section~\ref{sec:survey}.}.

\item One may alternatively use, instead of \eqref{eq:cour}, the non-skew-symmetric bracket
$$
((X,\alpha),(Y,\beta))\mapsto ([X,Y],\Lie_X\beta - i_Yd\alpha)
$$
for condition (ii); \eqref{eq:cour} is the skew-symmetrization of
this bracket, and a simple computation shows that both brackets
agree on sections of subbundles satisfying (i).
\end{itemize}
\end{remarks}

A subbundle $L\subset \TM$ satisfying (i) is sometimes referred to
as a \textit{lagrangian subbundle} of $\TM$ (in analogy with the
terminology in symplectic geometry), or as an \textit{almost Dirac
structure} on $M$. Condition (ii) is referred to as the
\textit{integrability condition}; by (i), condition (ii) can be
equivalently written as
\begin{equation}\label{eq:3tensor}
\SP{\Cour{a_1,a_2},a_3} \equiv 0,\;\;\; \forall \, a_1, a_2, a_3 \in
\Gamma(L).
\end{equation}
For any lagrangian subbundle $L\subset \TM$, the expression
\begin{equation}\label{eq:couranttensor}
\Upsilon_L(a_1,a_2,a_3):= \SP{\Cour{a_1,a_2},a_3},
\end{equation}
for $a_1, a_2, a_3 \in \Gamma(L)$, defines an element $\Upsilon_L\in
\Gamma(\wedge^3L^*)$ that we call the \textit{Courant tensor} of
$L$. Hence, for a lagrangian subbundle $L$, the integrability
condition (ii) is equivalent to the vanishing condition $\Upsilon_L
\equiv 0$.

\begin{example}
Any bivector field $\pi \in \Gamma(\wedge^2TM)$ defines a lagrangian
subbundle of $\mathbb{T}M$ given by the graph of \eqref{eq:pisharp},
\begin{equation}\label{eq:Lpi}
L_\pi =\{(\pi^\sharp(\alpha),\alpha)\,|\, \alpha\in T^*M\}.
\end{equation}
One may verify that, for $a_i = (\pi^\sharp(df_i),df_i)$, $i=1, 2,
3$,
$$
\Upsilon_{L_\pi}(a_1,a_2,a_3)=\SP{\Cour{a_1,a_2},a_3} =
\{f_1,\{f_2,f_3\}\} + c.p.
$$
So \eqref{eq:3tensor} holds, i.e., \eqref{eq:Lpi} is a Dirac
structure, if and only if $\pi$ is a Poisson bivector field. In
fact, Poisson structures can be identified with Dirac structures
$L\subset \mathbb{T}M$ with the additional property that
\begin{equation}\label{eq:transv1}
L\cap TM = \{0\}.
\end{equation}

Similarly, any 2-form $\omega\in \Omega^2(M)$ defines the lagrangian
subbundle
\begin{equation}\label{eq:Lomega}
L_{\omega}= \mathrm{graph}(\omega^\sharp) =
\{(X,\omega^\sharp(X))\;|\; X\in TM\} \subset \mathbb{T}M.
\end{equation}
In this case, for $a_i=(X_i,\omega^\sharp(X_i))$, $i=1,2,3$,
$$
\Upsilon_{L_\omega}(a_1,a_2,a_3) = \SP{\Cour{a_1,a_2},a_3} =
d\omega(X_1,X_2,X_3).
$$
So \eqref{eq:Lomega} is a Dirac structure if and only if $\omega$ is
presymplectic, and presymplectic structures are identified with
Dirac structures $L$ satisfying
\begin{equation}\label{eq:transv2}
L\cap T^*M = \{0\}.
\end{equation}
\end{example}

\begin{example}\label{ex:foliation}
Let $F\subseteq TM$ be a regular distribution, i.e., a vector
subbundle, and let $L= F\oplus F^\circ$, where $F^\circ\subseteq
T^*M$ is the annihilator of $F$. Clearly $L=L^\perp$, and one may
check that $L$ satisfies the integrability condition (ii) if and
only if $F$ is involutive, i.e., it is tangent to the leaves of a
regular foliation (by Frobenius' theorem). So one may view regular
foliations as particular cases of Dirac structures.
\end{example}

\begin{remark}\label{rem:complex}
The notion of Dirac structure may be naturally extended to the
complexification $\mathbb{T}M\otimes \mathbb{C}$. A natural example
arises as follows: consider an almost complex structure $J$ on $M$,
i.e., an endomorphism $J:TM\to TM$ satisfying $J^2=-1$, and let
$T_{1,0} \subset TM\otimes \mathbb{C}$ be its $+i$-eigenbundle.
Following the previous example, $L= T_{1,0}\oplus (T_{1,0})^\circ =
T_{1,0}\oplus T^{0,1}$ is a (complex) Dirac structure provided
$T_{1,0}$ is involutive, i.e., $J$ is integrable  as a complex
structure. An interesting class of complex Dirac structures that
includes this example is studied in \cite{Gua} (see
Section~\ref{sec:survey}).
\end{remark}

\begin{example}
Let $P$ be a manifold, and let $\omega_t$, $t\in \mathbb{R}$, be a
smooth family of closed 2-forms on $P$. This family determines a
Dirac structure $L$ on $M=P\times \mathbb{R}$ by
$$
L_{(x,t)}=\{((X,0),(i_X\omega_t,\gamma))\,|\, X\in T_xP,\, \gamma\in
\mathbb{R}\}\subset (T_xP\times \mathbb{R}) \oplus (T^*_xP \times
\mathbb{R}),
$$
where $(x,t)\in P\times \mathbb{R}$.
\end{example}

\begin{example}
Consider $M= \mathbb{R}^3$, with coordinates $(x,y,z)$, and let
$$
L = \mathrm{span}\SP{\left(\frac{\partial}{\partial y}, z dx\right),
\left(\frac{\partial}{\partial x}, -z dy\right), \left(0,dz\right
)}.
$$
For $z\neq 0$, condition \eqref{eq:transv1} is satisfied, and we see
that we can write $L$ as the graph of the Poisson structure $\pi =
\frac{1}{z}\frac{\partial}{\partial x}\wedge
\frac{\partial}{\partial y}$, with brackets
$$
\{x,y\}=\frac{1}{z},\;\; \{x,z\}=0,\;\; \{y,z\}=0.
$$
Note that $L$ is a {\em smooth Dirac structure}, despite the fact
that this Poisson structure is singular at $z=0$; the singularity
just reflects the fact that $L$ ceases to be the graph of a bivector
field (i.e., to satisfy \eqref{eq:transv1}) when $z=0$.
\end{example}

\begin{example}\label{ex:gauge}
Other examples of Dirac structure arise from the observation
\cite{SW} that the abelian group of closed 2-forms
$\Omega^2_{cl}(M)$ acts on $\mathbb{T}M$ preserving both
$\SP{\cdot,\cdot}$ and $\Cour{\cdot,\cdot}$ (more generally, the
group of vector-bundle automorphisms of $\TM$ preserving
$\SP{\cdot,\cdot}$ and $\Cour{\cdot,\cdot}$ is $\Diff(M)\ltimes
\Omega^2_{cl}(M)$, see e.g. \cite[Sec.~3]{Gua}); the action is given
by
$$
(X,\alpha) \stackrel{\tau_B}{\mapsto} (X,\alpha+i_XB),\;\;\; \mbox{
for } B\in \Omega^2_{cl}(M).
$$
As a result, this operation, called {\em gauge transformation} in
\cite{SW}, sends Dirac structures to Dirac structures. So one may
construct new Dirac structures from old with the aid of closed
2-forms. For example, given an involutive distribution $F\subseteq
TM$, then the associated Dirac structure of
Example~\ref{ex:foliation} may be modified by $B\in
\Omega^2_{cl}(M)$ to yield a new Dirac structure
$$
\{(X,\alpha)\,|\, X\in F,\; (\alpha -i_XB)|_F=0\}.
$$
\end{example}

More examples of Dirac structures coming from ``constraints'' in
Poisson manifolds will be discussed in Section~\ref{sec:back}.

\section{Properties of Dirac structures}
\label{sec:properties}

We now describe the main geometric features of Dirac structures,
generalizing those of presymplectic and Poisson manifolds recalled
in Section~\ref{sec:poisspre}. Specifically, we will see how the
null distribution and the Poisson algebra of admissible functions of
presymplectic manifolds, and the symplectic foliation of Poisson
manifolds, generalize to Dirac manifolds.

\subsection{Lie algebroid}

A \textit{Lie algebroid} is a vector bundle $A\to M$ equipped with a
Lie bracket $[\cdot,\cdot]$ on $\Gamma(A)$ and a bundle map (called
the \textit{anchor}) $\rho:A \to TM$ so that
$$
[u,f v]=(\Lie_{\rho(u)}f) v + f [u,v],
$$
for $u, v \in \Gamma(A)$ and $f\in C^\infty(M)$. A key property of a
Lie algebroid $A$ is that $\rho(A)\subseteq TM$ is an integrable
smooth distribution (possibly of non-constant rank); one refers to
the leaves of the associated foliation as \textit{leaves} or
\textit{orbits} of the Lie algebroid. More details on Lie algebroids
can be found e.g. in \cite{CW,DZ}.

For a Dirac structure $L$ on $M$, the vector bundle $L\to M$
inherits a Lie algebroid structure, with bracket on $\Gamma(L)$
given by the restriction of $\Cour{\cdot,\cdot}$  and anchor given
by restriction of $\pr_T$ to $L$:
\begin{equation}\label{eq:LAstr}
[\cdot,\cdot]_L := \Cour{\cdot,\cdot}|_{\Gamma(L)},\;\;\; \rho_L :=
\pr_T|_L: L\to TM.
\end{equation}
Note that the Jacobi identity of $[\cdot,\cdot]_L$ follows from
\eqref{eq:jacobifails} and \eqref{eq:3tensor}.

On a Poisson manifold $(M,\pi)$, viewed as a Dirac structure with
$L$ given by \eqref{eq:Lpi}, we have a vector-bundle isomorphism
$\pr_{T^*}|_L: L\to T^*M$. The Lie algebroid structure on $L$ can be
carried over to $T^*M$, and it coincides with the one mentioned in
Remark~\ref{rem:liealg}.

For a presymplectic manifold $(M,\omega)$, the anchor $\rho_L: L\to
TM$ is an isomorphism between the Lie algebroid structure of $L$ and
the canonical one on $TM$ (defined by the usual Lie bracket of
vector fields and anchor given by the identity).

\subsection{Presymplectic leaves and null distribution}
Since \eqref{eq:LAstr} endows any Dirac structure $L$ on $M$ with a
Lie algebroid structure, the distribution
\begin{equation}\label{eq:distrib}
R = \pr_T(L)\subseteq TM
\end{equation}
is integrable and defines a (singular) foliation on $M$. Each leaf
$\mathcal{O}$ of this foliation naturally inherits a closed 2-form
$\Omega_L \in \Omega^2(\mathcal{O})$, defined at each $x\in
\mathcal{O}$ by
$$
\Omega_L(X,Y)=\alpha(Y),  \;\; \mbox{ where } \; X,Y \in
T_x\mathcal{O}=R|_x \;\mbox{ and }\; (X,\alpha) \in L|_x.
$$
The fact that the formula for $\Omega_L$ does not depend on the
choice of $\alpha$ results from the observation that
$$
R^\circ = L\cap T^*M,
$$
while the smoothness of $\Omega_L$ follows from the existence of
smooth splittings for the surjective bundle map $\pr_T|_L:
L|_{\mathcal{O}}\to T\mathcal{O}$; the condition $d\Omega_L=0$
follows from the integrability condition of $L$. One refers to the
foliation defined by \eqref{eq:distrib} as the \textit{presymplectic
foliation} of $L$. In the case of a Poisson manifold, this is just
its symplectic foliation, while for a presymplectic manifold the
leaves are its connected components.

The distribution
\begin{equation}\label{eq:null}
K := L\cap TM \subseteq TM
\end{equation}
agrees, at each point, with the kernel of the leafwise 2-form
$\Omega_L$, so we refer to it as the \textit{null distribution} (or
\textit{kernel}) of the Dirac structure; as we saw for presymplectic
manifolds, this distribution is not smooth unless it is regular, in
which case it is automatically integrable, giving rise to a regular
foliation - the \textit{null foliation}. Note that the null
distribution is zero at all points if and only if $L$ is defined by
a Poisson structure (c.f. \eqref{eq:transv1}).

In analogy to what happens for Poisson manifolds, a Dirac structure
$L$ is completely determined by its presymplectic foliation; indeed,
at each point,
$$
L=\{(X,\alpha)\,|\, X\in \pr_T(L),\; \alpha|_{\pr_T(L)} =
i_X\Omega_L\}.
$$

The reader should have no problem to identify the distributions $K$,
$R$, and the presymplectic foliation in each example of
Section~\ref{sec:dirac}.

In the spirit of Section~\ref{sec:poisspre} we have, schematically,
the following equivalences:
\begin{eqnarray*}
\mbox{Nondegenerate Poisson structure }  & \rightleftharpoons&
\mbox{ Symplectic structure}\\
\mbox{Poisson structure} &  \rightleftharpoons & \mbox{ Symplectic foliation}\\
\mbox{Dirac structure} &  \rightleftharpoons& \mbox{ Presymplectic
foliation}
\end{eqnarray*}

\subsection{Hamiltonian vector fields and Poisson algebra}

The notion of hamiltonian vector field also extends to Dirac
structures. Let $L$ be a Dirac structure on $M$. Following the
discussion for presymplectic manifolds in Section~\ref{subsub:pre},
a function $f\in C^\infty(M)$ is called \textit{admissible} if there
is a vector field $X$ such that
$$
(X,df) \in L,
$$
in which case $X$ is called \textit{hamiltonian} relative to $f$.
Just as for presymplectic manifolds, $X$ is not uniquely determined,
as it may be modified by the addition of any vector field tangent to
the null distribution $K=L\cap TM$. The space of admissible
functions, denoted by $C^\infty_{adm}(M)$, is always a Poisson
algebra, with Poisson bracket defined as in \eqref{eq:pbrk1}.

When $K$ is regular, a function $f$ is admissible if and only if its
differential annihilates $K$, i.e., if $f$ is constant along the
leaves of the null foliation. In particular, if the null foliation
is simple, its leaf space acquires a Poisson structure (generalizing
the symplectic structure $\omega_{red}$ of Section~\ref{subsub:pre})
through the identification of its functions with admissible
functions on $M$.

In conclusion, Dirac structures naturally mix presymplectic and
Poisson features (controlled by the distributions $R$ and $K$):
presymplectic structures arise on their leaves, while Poisson
structures arise on their algebras of admissible functions.

\section{Morphisms of Dirac manifolds} \label{sec:morphisms}

As we saw in Section \ref{sec:poisspre}, symplectic structures are
equivalently described in terms of (nondegenerate) 2-forms or
bivector fields. As much as the two approaches are equivalent, they
naturally lead to different notions of morphism, reflecting the fact
that covariant and contravariant tensors have different functorial
properties with respect to maps.

Let $M_i$ be equipped with a symplectic form $\omega_i$, and
corresponding Poisson bivector field $\pi_i$, $i=1,2$. There are two
natural ways in which a map $\varphi: M_1\to M_2$ can preserve
symplectic structures: either by preserving symplectic 2-forms,
\begin{equation}\label{eq:sym1}
\varphi^*\omega_2 = \omega_1
\end{equation}
or by preserving Poisson bivector fields,
\begin{equation}\label{eq:sym2}
\varphi_*\pi_1 = \pi_2;
\end{equation}
by \eqref{eq:sym2} we mean that $\pi_1$ and $\pi_2$ are
$\varphi$-related: for all $x\in M_1$,
$$
(\pi_2)_{\varphi(x)}(\alpha,\beta) = (\pi_1)_x
(\varphi^*\alpha,\varphi^*\beta),\;\;\; \forall \, \alpha,\beta \in
T_x^*M_2.
$$

Conditions \eqref{eq:sym1} and \eqref{eq:sym2} are not equivalent.
Consider $\mathbb{R}^2$ with coordinates $(q^1,p_1)$ and symplectic
structure (written both as a 2-form and as a bivector field)
$$
\omega_{\mathbb{R}^2} = dq^1\wedge dp_1,\;\;\; \pi_{\mathbb{R}^2} =
\frac{\partial}{\partial{p_1}}\wedge\frac{\partial}{\partial{q^1}},
$$
and $\mathbb{R}^4$ with coordinates $(q^1,p_1,q^2,p_2)$ and
symplectic structure
$$
\omega_{\mathbb{R}^4} = dq^1\wedge dp_1 + dq^2\wedge dp_2,\;\;\;
\pi_{\mathbb{R}^4} =
\frac{\partial}{\partial{p_1}}\wedge\frac{\partial}{\partial{q^1}} +
\frac{\partial}{\partial{p_2}}\wedge \frac{\partial}{\partial{q^2}}.
$$
One readily verifies that the projection $\mathbb{R}^4 \to
\mathbb{R}^2$, $(q^1,p_1,q^2,p_2)\mapsto (q^1,p_1)$ satisfies
\eqref{eq:sym2}, but not \eqref{eq:sym1}, while the inclusion
$\mathbb{R}^2\hookrightarrow \mathbb{R}^4$, $(q^1,p_1)\mapsto
(q^1,p_1,0,0)$, satisfies \eqref{eq:sym1} but not \eqref{eq:sym2}.
More generally, note that the nondegeneracy of $\omega_1$ in
\eqref{eq:sym1} forces $\varphi$ to be an immersion, while the
nondegeneracy of $\pi_2$ in \eqref{eq:sym2} forces $\varphi$ to be a
submersion. The two conditions \eqref{eq:sym1} and \eqref{eq:sym2}
become equivalent when $\varphi$ is a local diffeomorphism. It is
also clear that \eqref{eq:sym1} naturally extends to a notion of
morphism between presymplectic manifolds, whereas \eqref{eq:sym2}
leads to the usual notion of \textit{Poisson map} between Poisson
manifolds.

As first noticed by A. Weinstein, the pull-back and push-forward
relations carry over to Dirac structures, leading to two distinct
notions of morphism for Dirac manifolds, generalizing
\eqref{eq:sym1} and \eqref{eq:sym2}.


\subsection{Pulling back and pushing forward}

In order to extend the notions of pullback and pushforward to the
realm of Dirac structures, it is convenient to start at the level of
linear algebra.

\subsubsection{Lagrangian relations}\label{subsub:linear}\
In the linear-algebra context, by a \textit{Dirac structure on a
vector space $V$} we will simply mean a subspace $L \subset V\oplus
V^*$ that is lagrangian, i.e., such that $L=L^\perp$  with respect
to the natural symmetric pairing
$$
\SP{(v_1,\alpha_1),(v_2,\alpha_2)}=\alpha_2(v_1) + \alpha_1(v_2).
$$
We also simplify the notation by writing $\mathbb{V}=V\oplus V^*$;
we denote by $\pr_V:\mathbb{V}\to V$ and $\pr_{V^*}:\mathbb{V}\to
V^*$ the natural projections.

Let $\varphi : V \to W$ be a linear map. We know that 2-forms
$\omega \in \wedge^2 W^*$ can be pulled back to $V$, while bivectors
$\pi \in \wedge^2 V$ can be pushed forward to $W$:
$$
\varphi^*\omega(v_1,v_2) = \omega(\varphi(v_1),\varphi(v_2)), \;\;\;
\varphi_*\pi(\beta_1,\beta_2) =
\pi(\varphi^*\beta_1,\varphi^*\beta_2),
$$
for $v_1,v_2 \in V$, $\beta_1, \beta_2 \in W^*$. These notions are
generalized to Dirac structures as follows. The {\em backward image}
of a Dirac structure $L_W \subset \mathbb{W}$ under $\varphi$ is
defined by
\begin{equation}\label{eq:back}
\mathfrak{B}_\varphi(L_W) =  \{ (v,\varphi^*\beta)\;|\;
(\varphi(v),\beta) \in L_W\} \subset V\oplus V^*,
\end{equation}
while the {\em forward image} of $L_V \in V\oplus V^*$ is
\begin{equation}\label{eq:forward}
\mathfrak{F}_\varphi(L_V) = \{ (\varphi(v),\beta)\;|\;
(v,\varphi^*\beta) \in L_V \} \subset W\oplus W^*.
\end{equation}

\begin{proposition}\label{prop:bfdirac}
Both $\mathfrak{B}_\varphi(L_W)$ and $\mathfrak{F}_\varphi(L_V)$ are
Dirac structures.
\end{proposition}

The proof will follow from Lemma~\ref{lem:lagrel} below. Note that
when $L_W$ is defined by $\omega\in \wedge^2W^*$, then
$\mathfrak{B}_\varphi(L_W)$ is the Dirac structure associated with
$\varphi^*\omega$ and, similarly, if $L_V$ is defined by $\pi\in
\wedge^2V$, then $\mathfrak{F}_\varphi(L_V)$ is the Dirac structure
corresponding to $\varphi_* \pi$.
So we also refer to the Dirac structures \eqref{eq:back} and
\eqref{eq:forward} as the \textit{pullback} of $L_W$ and the
\textit{pushforward} of $L_V$, respectively.

The pullback and pushforward operations \eqref{eq:back} and
\eqref{eq:forward} are not inverses of one another; one may verify
that
\begin{align}
& \label{eq:bf} \mathfrak{F}_\varphi(\mathfrak{B}_\varphi(L_W))=L_W
\mbox{ if and only if } \, \pr_W(L_W)\subseteq \varphi(V), \mbox{
and }\\
& \label{eq:bf2} \mathfrak{B}_\varphi(\mathfrak{F}_\varphi(L_V))=L_V
\mbox{ if and only if } \, \ker(\varphi)\subseteq L_V\cap V.
\end{align}
 As a result, $\mathfrak{F}_\varphi\circ \mathfrak{B}_\varphi = \Id$ if and only
if $\varphi$ is surjective, while $\mathfrak{B}_\varphi\circ
\mathfrak{F}_\varphi = \Id$ if and only if $\varphi$ is injective.

Pullback and pushforward of Dirac structures can be understood in
the broader context of composition of lagrangian relations (which is
totally analogous to the calculus of canonical relations in the
linear symplectic category, see \cite{GS}, \cite[Sec.~5]{BW}; see
also \cite{We10} and references therein), that we now recall. Let us
denote by $\overline{\mathbb{V}}$ the vector space
$\mathbb{V}=V\oplus V^*$ equipped with the pairing
$-\SP{\cdot,\cdot}$. The product $\mathbb{V}\times \mathbb{W}$
inherits a natural pairing from  $\mathbb{V}$ and $\mathbb{W}$. We
call a lagrangian subspace $L\subset \mathbb{V}\times
\overline{\mathbb{W}}$ a \textit{lagrangian relation} (from $W$ to
$V$).
Note that any Dirac structure $L \subset \mathbb{V}$ may be seen as
a lagrangian relation (either from $V$ to the trivial vector space
$\{0\}$, or the other way around), and any linear map $\varphi :
V\to W$ defines a lagrangian relation
\begin{equation}\label{eq:philagrel}
\Gamma_\varphi = \{ ((w,\beta),(v,\alpha)) \;|\; w=\varphi(v),\,
\alpha=\varphi^*(\beta) \} \subset {\mathbb{W}}\times
\overline{\mathbb{V}}.
\end{equation}

Lagrangian relations can be composed much in the same way as maps.
Consider vector spaces $U, V, W$ and lagrangian relations
$L_1\subseteq \mathbb{U}\times \overline{\mathbb{V}}$ and $L_2
\subset \mathbb{V}\times \overline{\mathbb{W}}$. The
\textit{composition} of $L_1$ and $L_2$, denoted by $L_1\circ L_2
\subset \mathbb{U}\times \overline{\mathbb{W}}$, is defined by
$$
L_1\circ L_2 := \{((u,\alpha),(w,\gamma))\;|\; \exists \,
(v,\beta)\in \mathbb{V}\; s.t.\; (u,\alpha, v ,\beta)\in L_1,\;
(v,\beta,w,\gamma)\in L_2\}.
$$
\begin{lemma}\label{lem:lagrel}
The composition $L_1\circ L_2$ is a lagrangian relation.
\end{lemma}
\begin{proof}
To verify that $(L_1\circ L_2)^\perp = L_1\circ L_2$, consider the
subspace
\begin{equation}\label{eq:C}
\cC := \mathbb{U} \times \Delta \times \overline{\mathbb{W}} \subset
\mathbb{U} \times \overline{\mathbb{V}}\times {\mathbb{V}} \times
\overline{\mathbb{W}},
\end{equation}
where $\Delta$ is the diagonal in $\overline{\mathbb{V}}\times
{\mathbb{V}}$. Then
$$
\cC^\perp = \{0\} \times \Delta \times \{0\} \subseteq \cC,
$$
and $L_1\circ L_2$ is the image of $L_1\times L_2 \subset \mathbb{U}
\times \overline{\mathbb{V}}\times {\mathbb{V}}  \times
\overline{\mathbb{W}}$ in the quotient $\cC/\cC^\perp =
\mathbb{U}\times \overline{\mathbb{W}}$:
\begin{equation}\label{eq:compform}
 L_1\circ L_2 =
\frac{(L_1\times L_2)\cap \cC + \cC^\perp}{\cC^\perp}
\end{equation}

The projection $p: \cC \to \cC/\cC^\perp$ satisfies
$\SP{p(c_1),p(c_2)}=\SP{c_1,c_2}$ for all $c_1, c_2 \in \cC$. So
$p(c_1) \in (L_1\circ L_2)^\perp$, i.e., $\SP{p(c_1),p(c_2)}=0$ for
all $p(c_2)\in L_1\circ L_2$, if and only if $\SP{c_1,c_2}=0$ for
all $c_2 \in (L_1\times L_2)\cap \cC + \cC^\perp$, i.e.,
$$
c_1 \in ((L_1\times L_2)\cap \cC + \cC^\perp)^\perp.
$$
Now note that\footnote{Here we use that, for subspaces $A, B$ of any
vector space equipped with a symmetric pairing with split signature,
we have $(A+B)^\perp = A^\perp \cap B^\perp$ and $(A\cap B)^\perp =
A^\perp + B^\perp$.}
\begin{align*}
((L_1\times L_2)\cap \cC + \cC^\perp)^\perp & = (L_1\times L_2 +
\cC^\perp)\cap \cC \\& = (L_1\times L_2)\cap \cC + \cC^\perp,
\end{align*}
from where we conclude that $(L_1\circ L_2)^\perp = L_1\circ L_2$.
\end{proof}

We can now conclude the proof of Prop.~\ref{prop:bfdirac}.

\begin{proof}(of Prop.~\ref{prop:bfdirac}) Let us view a Dirac structure $L_V$ on $V$ as a lagrangian
relation $L_V \subset \mathbb{V}\times \{0\}$; similarly, we regard
a Dirac structure $L_W$ on $W$ as a lagrangian relation $L_W\subset
\{0\}\times \overline{\mathbb{W}}$. Then one immediately verifies
that the pushforward and pullback with respect to a linear map
$\varphi: V\to W$ are given by compositions of relations with
respect to the lagrangian relation $\Gamma_\varphi$ in
\eqref{eq:philagrel}:
\begin{equation}\label{eq:bfcomp}
\mathfrak{B}_\varphi(L_W) = L_W\circ\Gamma_\varphi, \;\;\;
\mathfrak{F}_\varphi(L_V) = \Gamma_\varphi\circ L_V.
\end{equation}
If follows from Lemma~\ref{lem:lagrel} that \eqref{eq:back} and
\eqref{eq:forward} are Dirac structures.
\end{proof}

Before moving on to Dirac structures on manifolds, we discuss for
later use how the composition expression in \eqref{eq:compform} may
be simplified for the particular cases in \eqref{eq:bfcomp}. For the
composition $\mathfrak{B}_\varphi(L_W) = L_W\circ \Gamma_\varphi$,
we have (c.f. \eqref{eq:C})
$$\cC=
\Delta_{\mathbb{W}}\times \overline{\mathbb{V}} \subset
\overline{\mathbb{W}}\times {\mathbb{W}}\times
\overline{\mathbb{V}},
$$
where $\Delta_{\mathbb{W}}$ is the diagonal subspace in
$\overline{\mathbb{W}}\times \mathbb{W}$, so that $\cC^\perp =
\Delta_{\mathbb{W}}\times \{0\}$; then, by \eqref{eq:compform},
$\mathfrak{B}_\varphi(L_W)$ is the image of $(L_W\times
\Gamma_\varphi)\cap \cC$ under the projection $p: \cC \to
\cC/\cC^\perp = \overline{\mathbb{V}}$. We may write $p$ as a
composition of two maps: the isomorphism $p_1: \cC
\stackrel{\sim}{\to} \mathbb{W}\times {\mathbb{V}}$, $(w,w,v)\mapsto
(w,v)$, followed by the projection $p_2: \mathbb{W}\times
{\mathbb{V}} \to {\mathbb{V}}$. Noticing that $p_1((L_W\times
\Gamma_\varphi)\cap \cC) = (L_W\times {\mathbb{V}})\cap
\Gamma_\varphi$, we can write $\mathfrak{B}_\varphi(L_W) =
p_2((L_W\times {\mathbb{V}})\cap \Gamma_\varphi)$. One also checks
that
$$
\ker(p_2|_{(L_W\times {\mathbb{V}})\cap \Gamma_\varphi)}) =
\{((0,\beta),(0,0))\,|\, (0,\beta) \in L_W,\, \varphi^*\beta = 0\} =
(\ker(\varphi^*)\cap L_W) \times \{0\}.
$$
As a result, we obtain

\begin{proposition}
The following is a short exact sequence:
\begin{equation}\label{eq:short1}
(\ker(\varphi^*)\cap L_W) \times \{0\} \hookrightarrow (L_W\times
{\mathbb{V}})\cap \Gamma_\varphi \twoheadrightarrow
\mathfrak{B}_\varphi(L_W),
\end{equation}
where the second map is the projection $\mathbb{W}\times
\mathbb{V}\to \mathbb{V}$. Similarly, we have the short exact
sequence
\begin{equation}\label{eq:short2}
\{0\}\times (\ker(\varphi)\cap L_V) \hookrightarrow
\Gamma_\varphi\cap (\mathbb{W}\times L_V) \twoheadrightarrow
\mathfrak{F}_\varphi(L_V),
\end{equation}
where the second map is the projection $\mathbb{W}\times
\mathbb{V}\to \mathbb{W}$.
\end{proposition}

\subsubsection{Morphisms on Dirac manifolds}\
We now transfer the notions of pullback and pushforward of Dirac
structures to the manifold setting, i.e., the role of $\mathbb{V}$
will be played by $\TM$.

Let $L_M \subset \mathbb{T}M$ and $L_N\subset \mathbb{T}N$ be almost
Dirac structures, and $\varphi: M \to N$  a smooth map. Analogously
to \eqref{eq:philagrel}, we consider the lagrangian relation
$$
\Gamma_{\varphi}=\{((Y,\beta),(X,\alpha))\;|\; Y=d\varphi(X),\;
\alpha = \varphi^*\beta \} \subset \varphi^*\TN\oplus \overline{\TM}
$$
where $\varphi^*\TN \to M$ is the pull-back bundle of $\TN$ by
$\varphi$, and we define, as in \eqref{eq:bfcomp}, the
\textit{backward image} of $L_N$ by
\begin{equation}\label{eq:bimage}
\mathfrak{B}_\varphi(L_N) = (\varphi^*L_N) \circ \Gamma_\varphi
\subset \TM,
\end{equation}
and the \textit{forward image} of $L_M$ by
\begin{equation}\label{eq:fimage}
\mathfrak{F}_\varphi(L_M) = \Gamma_\varphi\circ L_M \subset
\varphi^*\TN.
\end{equation}
 By what we saw in Section~\ref{subsub:linear},
$\mathfrak{B}_\varphi(L_N)$ and $\mathfrak{F}_\varphi(L_M)$ define
lagrangian distributions in $\TM$ and $\varphi^*\TN$, respectively;
the issue of whether they are smooth will be addressed in the next
section.

We call $\varphi: M\to N$ a \textit{backward Dirac map} (or simply
\textit{b-Dirac}) when $L_M$ coincides with
$\mathfrak{B}_\varphi(L_N)$,
\begin{equation}\label{eq:pull}
(L_M)_{x} = \mathfrak{B}_\varphi(L_N)_x =  \{
(X,d\varphi^*(\beta))\;|\; (d\varphi(X),\beta)\in (L_N)_{\varphi(x)}
\}, \;\;\; \forall \, x \in M,
\end{equation}
and we call it a \textit{forward Dirac map} (or simply
\textit{f-Dirac}) when $\varphi^*L_N$ (the pullback of $L_N$ to $N$
as a vector bundle) coincides with $\mathfrak{F}_\varphi(L_M)$,
i.e., $\forall x\in M$,
\begin{equation}\label{eq:push}
(L_N)_{\varphi(x)} = \mathfrak{F}_\varphi(L_M)_x  = \{
(d\varphi(X),\beta)\;|\; (X,d\varphi^*(\beta))\in (L_M)_x \}.
\end{equation}
As already observed, \eqref{eq:pull} generalizes the pullback of
2-forms, while \eqref{eq:push} extends the notion of two bivector
fields being $\varphi$-related.

If $\varphi: M\to M$ is a diffeomorphism, then it is f-Dirac if and
only if it is b-Dirac, and these are both equivalent to the
condition that the canonical lift of $\varphi$ to $\mathbb{T}M$
preserves the subbundle $L_M\subset \mathbb{T}M$:
\begin{equation}\label{eq:diracdiff}
\Phi := (d\varphi,(d\varphi^{-1})^*):\mathbb{T}M\to
\mathbb{T}M,\;\;\; \Phi(L_M)=L_M.
\end{equation}
We refer to $\varphi$ as a \emph{Dirac diffeomorphism}.

\begin{remark}\label{rem:transv}
Forward Dirac maps $\varphi: M\to N$ satisfying the additional
transversality condition $\ker(d\varphi)\cap L_M=\{0\}$ are studied
in \cite{ABM,bcwz,BC} in the context of actions and momentum maps.
\end{remark}

\subsection{Clean intersection and smoothness issues}
In the previous subsection, we used pullback and pushforward to
obtain two ways to relate Dirac structures on $M$ and $N$ by a map
$\varphi: M\to N$: via \eqref{eq:pull} and \eqref{eq:push}. We now
discuss the possibility of using pullback and pushforward to {\it
transport} Dirac structures from one manifold to the other. In other
words, we know that the backward (resp. forward) image of a Dirac
structure \eqref{eq:bimage} (resp. \eqref{eq:fimage}) gives rise to
a lagrangian distribution in $\TM$ (resp. $\varphi^*\TN$), and the
issue is whether this distribution fits into a smooth bundle
defining a Dirac structure. Note that while the lagrangian
distribution generated by the backward image \eqref{eq:bimage} is
defined over all of $M$, the forward image \eqref{eq:fimage} only
defines the distribution over points in the image of $\varphi$.
(This reflects the well-known fact that, while any map $\varphi:M\to
N$ induces a map $\Omega^\bullet(N)\to \Omega^\bullet(M)$ by
pullback, there is in general no induced map from
$\mathcal{X}^\bullet(M)$ to $\mathcal{X}^\bullet(N)$.) This makes
the discussion for backward images a bit simpler, so we treat them
first.

\subsubsection{Backward images}\

Let $L_N$ be an almost Dirac structure on $N$, and consider a map
$\varphi:M\to N$. As we already remarked, \eqref{eq:pull} defines,
at each point $x \in M$, a lagrangian subspace of $T_xM\oplus
T^*_xM$. This lagrangian distribution does \textit{not} necessarily
fit into a smooth subbundle of $\TM$, as illustrated by the
following example.

\begin{example}\label{ex:pullns} Let $\varphi: \mathbb{R} \to \mathbb{R}^2$ be the
inclusion $x \mapsto (x,0)$, and consider the Dirac structure on
$\mathbb{R}^2$ defined by the bivector field
$$
\pi = x\frac{\partial}{\partial x}\wedge\frac{\partial}{\partial y}.
$$
We consider its backward image to the $x$-axis via $\varphi$.
Whenever $x\neq 0$, $\pi$ is nondegenerate, and hence corresponds to
a 2-form; its pullback is then a 2-form on $\mathbb{R}-\{0\}$, which
must therefore vanish identically (by dimension reasons). So the
backward image of $\pi$ on $\mathbb{R}$ is $T_x\mathbb{R} \subset
T_x\mathbb{R}\oplus T_x^*\mathbb{R}$ for $x\neq 0$. For $x=0$, one
can readily compute the backward image to be $T_x^*\mathbb{R}
\subset T_x\mathbb{R}\oplus T_x^*\mathbb{R}$, so the resulting
family of lagrangian subspaces is not a smooth vector bundle over
$\mathbb{R}$.
\end{example}

As we now see, a suitable clean-intersection assumption guarantees
the smoothness of backward images. For a smooth map $\varphi: M\to
N$, let $(d\varphi)^*: \varphi^*T^*N \to T^*M$ be the dual of the
tangent map $d\varphi: TM \to \varphi^*TN$. For an almost Dirac
structure $L_N$ on $N$, we denote by $\varphi^* L_N$ the vector
bundle over $M$ obtained by the pull-back of $L_N$ \textit{as a
vector bundle}.

\begin{proposition}\label{prop:pull}
Let $\varphi:M\to N$ be a smooth map, and let $L_N$ be an almost
Dirac structure on $N$. If $\ker((d\varphi)^*)\cap \varphi^*L_N$ has
constant rank, then the backward image $\mathfrak{B}_\varphi(L_N)
\subset \TM$ defines a smooth lagrangian subbundle, i.e., it is an
almost Dirac structure on $M$. If $\mathfrak{B}_\varphi(L_N)$ is a
smooth bundle and $L_N$ is integrable, then
$\mathfrak{B}_\varphi(L_N)$ is integrable, i.e., it is a Dirac
structure.
\end{proposition}

\begin{proof}
The backward image of $L_N$ fits into a short  exact sequence that
is a pointwise version of \eqref{eq:short1}:
\begin{equation}\label{eq:exactb}
\ker((d\varphi)^*)\cap \varphi^*L_N \to (\varphi^*L_N\oplus
{\TM})\cap \Gamma_\varphi \to \mathfrak{B}_\varphi(L_N),
\end{equation}
where the first map is the inclusion $\varphi^*\TN\to
\varphi^*\TN\oplus {\TM}$, and the second is the projection of
$\varphi^*\TN\oplus {\TM}$ onto ${\TM}$. Since
$\mathfrak{B}_\varphi(L_N)$ has constant rank, it is clear that
$(\varphi^*\TN\oplus {\TM})\cap \Gamma_\varphi$ has constant rank if
and only if so does $\ker((d\varphi)^*)\cap \varphi^*L_N$; in this
case both are smooth vector bundles, and hence so is
$\mathfrak{B}_\varphi(L_N)$.

For the assertion about integrability, we compare the tensors
$\Upsilon_M$ and $\Upsilon_N$ associated with the almost Dirac
structures $\mathfrak{B}_\varphi(L_N)$ and $L_N$, respectively (see
\eqref{eq:couranttensor}). Let us call (local) sections
$a=(X,\alpha)$ of $\TM$ and $b=(Y,\beta)$ of $\TN$
\textit{$\varphi$-related} if the vector fields $X$ and $Y$ are
$\varphi$-related and $\alpha = \varphi^* \beta$. If now
 $a_i=(X_i,\alpha_i)$,
$b_i=(Y_i,\beta_i)$, $i=1,2,3$, are sections of $\TM$ and $\TN$,
respectively, such that $a_i$ and $b_i$ are $\varphi$-related, then
a direct computation shows that
\begin{equation}\label{eq:3tensorrel}
\SP{a_1, \Cour{a_2,a_3}} = \varphi^* \SP{b_1,\Cour{b_2,b_3}}.
\end{equation}

Suppose that $x_0\in M$ is such that
\begin{itemize}
\item[(i)] $d\varphi$ has constant rank around $x_0$, and
\item[(ii)] $\ker((d\varphi)^*)\cap \varphi^* L_N$ has constant rank
around $x_0$.
\end{itemize}
By (i), we know that locally, around $x_0$ and $\varphi(x_0)$, the
map $\varphi$ looks like
\begin{equation}\label{eq:model}
(x^1,\ldots,x^m)\stackrel{\varphi}{\mapsto}
(x^1,\ldots,x^k,0,\ldots,0),
\end{equation}
where $k$ is the rank of $d\varphi$ around $x$. In this local model,
$M=\{(x^1,\ldots,x^m)\}$, $N=\{(x^1,\ldots,x^n)\}$, and we identify
the submanifolds of $M$ and $N$ defined by $x^i=0$ for $i>k$ with
$S=\{(x^1,\ldots,x^k)\}$.

\smallskip

\noindent{\bf Claim:} \emph{For any $a_{x_0} \in
\mathfrak{B}_{\varphi}(L_N)|_{x_0}$, one can find local sections $a$
of $\mathfrak{B}_{\varphi}(L_N)$ and $b$ of $L_N$ so that
$a|_{x_0}=a_{x_0}$, and $a$ and $b$ are $\varphi$-related.}

\smallskip

To verify the claim, first extend $a_{x_0}$ to a local section
$a_S=(X_S,\alpha_S)$ of $\mathfrak{B}_{\varphi}(L_N)|_S$; using (ii)
and \eqref{eq:exactb}, one can then find a local section
$b_S=(Y_S,\beta_S)$ of $L_N|_S$ so that $(b_S,a_S)$ is a section of
$\Gamma_{\varphi}|_S$, i.e., $d\varphi(X_S)=Y_S$ and
$\alpha_S=\varphi^*\beta_S$. Using the local form \eqref{eq:model},
one may naturally extend $X_S$ to a vector field $X$ on $M$
satisfying $d\varphi(X) =Y_S$. It follows that $a=
(X,\varphi^*\beta_S)$ is a local section of
$\mathfrak{B}_{\varphi}(L_N)$ and, if $b$ is any local section of
$L_N$ extending $b_S$, $a$ and $b$ are $\varphi$-related. This
proves the claim.

Given arbitrary $(a_i)_{x_0} \in
\mathfrak{B}_{\varphi}(L_N)|_{x_0}$, $i=1,2,3$, by the previous
claim we can extend them to local sections $a_i$ of
$\mathfrak{B}_{\varphi}(L_N)$ and find local sections $b_i$ of
$L_N$, so that $a_i$ and $b_i$ are $\varphi$-related. From
\eqref{eq:3tensorrel}, we have
$$
\Upsilon_M((a_1)_{x_0},(a_2)_{x_0},(a_3)_{x_0}) =
\Upsilon_M(a_1,a_2,a_3)(x_0) = \Upsilon_N(b_1,b_2,b_3)(\varphi(x_0))
= 0,
$$
since $L_N$ is integrable. It directly follows that $\Upsilon_M$
vanishes at all points $x\in M$ satisfying (i) and (ii). Since these
points form an open dense subset in $M$, we conclude that
$\Upsilon_M\equiv 0$, as desired.

\end{proof}

We refer to the condition that $\ker((d\varphi)^*)\cap \varphi^*L_N$
has constant rank as the \textit{clean-intersection condition} (c.f.
\cite[Sec.~5]{BW}). By \eqref{eq:transv2}, it always holds when
$L_N$ is defined by a 2-form on $N$.

\begin{example}\label{ex:subm}
Let $L$ be a Dirac structure on $M$, and let $\varphi: C
\hookrightarrow M$ be the inclusion of a submanifold. In this case,
$\ker(d\varphi)^*=TC^\circ$, so the clean-intersection condition is
that $L\cap TC^\circ$ has constant rank. This guarantees that $C$
inherits a Dirac structure by pullback, explicitly given (noticing
that $(L|_C\oplus \mathbb{T}C)\cap \Gamma_{\varphi}$ projects
isomorphically onto $L\cap (TC\oplus T^*M|_C)$ in \eqref{eq:exactb})
by
$$
\mathfrak{B}_\varphi(L) = \{(X,\varphi^*\beta)\,|\, (X,\beta)\in L
\}\cong \frac{L\cap (TC \oplus T^*M|_C)}{L\cap TC^\circ} \subset
TC\oplus T^*C,
$$
as originally noticed in \cite{courant}.
\end{example}

\subsubsection{Forward images}\

Let $L_M$ be an almost Dirac structure on $M$ and $\varphi:M\to N$
be a smooth map. The forward image \eqref{eq:fimage} of $L_M$ by
$\varphi$ defines a lagrangian distribution in $\varphi^* \TN$. The
question of whether it is a Dirac structure on $N$ involves two
issues: first, as in the discussion of backward images, whether this
distribution  fits into a smooth subbundle of $\varphi^* \TN$;
second, if this subbundle determines a subbundle of $\TN$. The first
issue leads to a clean-intersection condition, analogous to the one
for backward images, while the second issue involves an additional
invariance condition with respect to the map $\varphi$. This
invariance condition is of course already necessary when one
considers the pushforward of bivector fields: if $\varphi: M\to N$
is a surjective submersion, for the pushforward of a bivector field
$\pi \in \Gamma(\wedge^2 TM)$ to be well defined, one needs that
$\pi$ satisfies, for all $\alpha,\beta\in T^*_yN$,
$$
\pi_x(\varphi^*\alpha,\varphi^*\beta) =
\pi_{x'}(\varphi^*\alpha,\varphi^*\beta)
$$
whenever $\varphi(x)=\varphi(x')=y$. We extend this invariance to
the context of Dirac structures as follows: we say that $L_M$ is
\textit{$\varphi$-invariant} if the right-hand side of
\eqref{eq:push} is invariant for all $x$ on the same
$\varphi$-fibre.

The next example illustrates a situation where a Dirac structure
satisfies the invariance condition, but its forward image is not a
smooth bundle.

\begin{example}
Consider the 2-form
$$
\omega = x dx\wedge dy
$$
on $\mathbb{R}^2$, independent of $y$, and let $\varphi:
\mathbb{R}^2\to \mathbb{R}$ be the projection $(x,y)\mapsto x$. For
$x\neq 0$, its forward image is $T^*_x\mathbb{R}\subset
T_x\mathbb{R}\oplus T^*_x\mathbb{R}$ (since the 2-form is
nondegenerate at these points, so we may consider the pushforward of
the corresponding bivector field), while for $x=0$ it is
$T_x\mathbb{R}\subset T_x\mathbb{R}\oplus T^*_x\mathbb{R}$.
\end{example}

The following result is parallel to Prop.~\ref{prop:pull}:

\begin{proposition}\label{prop:push}
Let $\varphi:M\to N$ be a surjective submersion and $L_M$ be an
almost Dirac structure on $M$. If $\ker(d\varphi)\cap L_M$ has
constant rank, then the forward image of $L_M$ by $\varphi$ is a
smooth lagrangian subbundle of $\varphi^*\mathbb{T}N$. If the
forward image of $L_M$ is a smooth bundle in $\varphi^*\mathbb{T}N$
and $L_M$ is $\varphi$-invariant, then it defines an almost Dirac
structure on $M$, which is integrable provided $L_M$ is.
\end{proposition}

\begin{proof}
Similarly to the proof of Prop.~\ref{prop:pull}, one has the short
exact sequence (c.f. \eqref{eq:short2})
\begin{equation}\label{eq:exactf}
\ker(d\varphi)\cap L_M \to \Gamma_\varphi \cap (\varphi^* \TN\oplus
L_M) \to \mathfrak{F}_\varphi(L_M),
\end{equation}
where the first map is the inclusion $\TM\to \varphi^*\TN\oplus \TM$
and the second is the projection $\varphi^*\TN\oplus \TM\to
\varphi^*\TN$. Since the rank of $\mathfrak{F}_\varphi(L_M)$ is
constant, one argues as in the proof of Prop.~\ref{prop:pull}:  the
constant-rank condition for $\ker(d\varphi)\cap L_M$ guarantees that
$\Gamma_\varphi \cap (\varphi^* \TN\oplus L_M)$ has constant rank,
hence $\mathfrak{F}_\varphi(L_M)\subset \varphi^*\TN$ is a
(lagrangian) vector subbundle.

The $\varphi$-invariance of $L_M$ guarantees that
$\mathfrak{F}_\varphi(L_M)$ determines a lagrangian subbundle of
$\TN$, that we denote by $L_N$, given by the image of
$\mathfrak{F}_\varphi(L_M)$ under the natural projection
$\varphi^*\TN \to \TN$; one may also verify that
$\mathfrak{F}_\varphi(L_M)=\varphi^*L_N$ (the pullback of $L_N$ as a
vector bundle).

The assertion about integrability is verified as in
Prop.~\ref{prop:pull}. If $x\in M$ is such that $\ker(d\varphi)\cap
L_M$ has constant rank around it, for any local section $b$ of $L_N$
around $\varphi(x)$ one can choose a local section $a$ of $L_M$
around $x$ (by splitting \eqref{eq:exactf}) which is
$\varphi$-related to $b$. Given arbitrary $(b_y)_i \in L_N|_y$,
$i=1,2,3$, where $y=\varphi(x)$, we extend them to local sections
$b_i$ of $L_N$, and take local sections $a_i$ of $L_M$ so that $a_i$
and $b_i$ are $\varphi$-related. Denoting by $\Upsilon_N$ and
$\Upsilon_M$ the Courant tensors of $L_N$ and $L_M$, respectively,
we have
$$
\Upsilon_N((b_y)_1,(b_y)_2,(b_y)_3)=\Upsilon_N(b_1,b_2,b_3)(y)=\Upsilon_M(a_1,a_2,a_3)(x)=0,
$$
assuming that $L_N$ is integrable. Using the fact that points $x\in
M$ around which $\ker(d\varphi)\cap L_M$ has constant rank form an
open dense subset of $M$, we conclude that $\Upsilon_N\equiv 0$.
\end{proof}

\begin{example}\label{ex:push}
Let $M$ be equipped with a Dirac structure $L$, and suppose that
$\varphi: M\to N$ is a surjective submersion whose fibres are
connected and such that
$$
\ker(d\varphi) \subseteq L\cap TM = K.
$$
Then $\ker(d\varphi)\cap L = \ker(d\varphi)$ has constant rank. We
claim that $L$ is automatically $\varphi$-invariant. To verify this
fact, the following is the key observation: given a vector field $Z$
(assume it to be complete, for simplicity), its flow $\psi^t_Z$
preserves $L$ (i.e., it is a family of Dirac diffeomorphisms, see
\eqref{eq:diracdiff}), if and only if
\begin{equation}\label{eq:inv}
\Lie_Z(X,\alpha)=([Z,X],\Lie_Z\alpha) \in \Gamma(L),\;\;\; \mbox{
for all }\, (X,\alpha)\in \Gamma(L).
\end{equation}

Since $\varphi$ is a surjective submersion, it is locally a
projection, so any $x\in M$ admits a neighborhood in which, for any
other $x'$ on the same $\varphi$-fibre, one can find a (compactly
supported) vector field $Z$, tangent to $\ker(d\varphi)$, whose flow
(for some time $t$) takes $x$ to $x'$. Note that \eqref{eq:inv}
holds, as the right-hand side of this equation is
$\Cour{(Z,0),(X,\alpha)}$, which must lie in $\Gamma(L)$ as a result
of the integrability of $L$ (since $Z$ is a section of
$\ker(d\varphi) \subseteq L\cap TM$). So the (time-$t$) flow of $Z$
defines a Dirac diffeomorphism taking $x$ to $x'$. Since $\varphi$
has connected fibres, it follows that one can find a Dirac
diffeomorphism (by composing finitely many flows) $\psi$ taking $x$
to any $x'$ on the same $\varphi$-fibre. Using $\psi$, we see that
$$
\mathfrak{F}_\varphi(L|_{x'})=\mathfrak{F}_\varphi(\mathfrak{F}_\psi(L|_x))=
\mathfrak{F}_{\varphi\circ\psi}(L|_{x'})=
\mathfrak{F}_{\varphi}(L|_x),
$$
so $L$ is $\varphi$-invariant.

It follows that the forward image of $L$, denoted by $L_N$, defines
a Dirac structure on $N$. Moreover, $L_N\cap TN = d\varphi(K)$. A
particular case of this situation is when the null foliation of $L$
is simple, $N$ is the leaf space, $\varphi:M\to N$ is the natural
projection, and $K=\ker(d\varphi)$. In this case, $L_N\cap
TN=\{0\}$, so $L_N$ is given by a bivector field $\pi_N$, determined
by
$$
\{f,g\}_N(\varphi(x)) = dg(\pi_N^\sharp(df))(\varphi(x)) =
dg(d\varphi(X_{\varphi^*f}))(\varphi(x))=
\{\varphi^*f,\varphi^*g\}(x).
$$
So $\pi_N$ is the Poisson structure on $N$ arising from the Poisson
bracket on admissible functions (in this case, these agree with the
basic functions relative to the null foliation) on $M$ via the
identification
$$
\varphi^*: C^\infty(N)\to C^\infty(M)_{adm}.
$$
Note that $\varphi$ is f-Dirac by construction, but it is also
b-Dirac by \eqref{eq:bf2}.
\end{example}

\begin{remark}\label{rem:group}
Consider a Lie group $G$ acting on $M$ preserving the Dirac
structure $L$, i.e., an action by Dirac diffeomorphisms. If the
action is free and proper, the quotient map $\varphi: M\to N=M/G$ is
a surjective submersion and $L$ is $\varphi$-invariant. In this case
the clean-intersection assumption, guaranteeing that the forward
image of $L$ on $N$ is a Dirac structure, is that the distribution
$L\cap D$ has constant rank, where $D$ is the distribution tangent
to the $G$-orbits. Extensions of this discussion to non-free actions
have been considered e.g. in \cite{jotzratiu}.
\end{remark}

\section{Submanifolds of Poisson
manifolds and constraints}\label{sec:back}

We now discuss Dirac manifolds arising as submanifolds of Poisson
manifolds, analogous to the presymplectic structures inherited by
submanifolds of symplectic manifolds. The discussion illustrates how
Dirac structures provide a convenient setting for the description of
the intrinsic geometry of constraints in Poisson phase spaces.

As pointed out in Example~\ref{ex:subm}, submanifolds of Dirac
manifolds inherit Dirac structures (modulo a cleanness issue) via
pullback. We will focus on the Dirac structures inherited by
submanifolds $\iota: C\hookrightarrow M $ of a given {\it Poisson
manifold} $(M,\pi)$. We denote by $L$ the Dirac structure on $M$
defined by $\pi$. Then
\begin{equation}\label{eq:bpi}
\mathfrak{B}_\iota(L) = \{(X,\iota^*\beta)\in \mathbb{T}C\,|\,
X=\pi^\sharp(\beta)\in TC,\, \beta \in T^*M \}.
\end{equation}
We will focus on the case where this is a smooth bundle, hence a
Dirac structure on $C$. Note that the clean-intersection condition
for this pullback is that $\ker(\pi^\sharp)\cap TC^\circ$ has
constant rank, which is equivalent to $\pi^\sharp(TC^\circ) \subset
TM|_C$ having constant rank, as a result of the short exact sequence
$\ker(\pi^\sharp)\cap TC^\circ \to TC^\circ \to
\pi^\sharp(TC^\circ)$.

The null distribution of $\mathfrak{B}_\iota(L)$ is
\begin{equation}\label{eq:bkker}
TC\cap \mathfrak{B}_\iota(L) = TC\cap \pi^\sharp(TC^\circ).
\end{equation}
We present a description of the Poisson bracket on admissible
functions on $C$ in terms of the Poisson structure on $M$.

\medskip
\paragraph{\bf The induced Poisson bracket on admissible functions}

Assume that a function $f$ on $C$ admits a local extension $\hat{f}$
to $M$ so that
\begin{equation}\label{eq:ext}
d\hat{f}(\pi^\sharp(TC^\circ))=0.
\end{equation}
Then $X_{\hat{f}}=\pi^\sharp(d\hat{f})$ satisfies
$\beta(X_{\hat{f}})=0$ for all $\beta\in TC^\circ$; i.e.,
$X_{\hat{f}}|_C\in TC$. It is then clear from \eqref{eq:bpi} that
$$
(X_{\hat{f}}|_C, df) \in \mathfrak{B}_\iota(L),
$$
i.e., $X_{\hat{f}}|_C$ is a hamiltonian vector field for $f$.

Recall that, if the null distribution \eqref{eq:bkker} has constant
rank, then  a function $f$ on $C$ is admissible for
$\mathfrak{B}_\iota(L)$ if and only if $df$ vanishes on $TC\cap
\pi^\sharp(TC^\circ)$. If we assume additionally that
$\pi^\sharp(TC^\circ)$ has constant rank, then any admissible
function $f$ on $C$ can be extended to a function $\hat{f}$ on $M$
satisfying \eqref{eq:ext}. Denoting by $\{\cdot,\cdot\}_C$ the
Poisson bracket on admissible functions, we have
\begin{equation}\label{eq:brhat}
\{f,g\}_C = X_f g = X_{\hat{f}}|_C g = X_{\hat{f}} \hat{g} |_C,
\end{equation}
i.e., the Poisson brackets on $C$ and $M$ are related by
\begin{equation}\label{eq:brks}
\{f,g\}_C = \{\hat{f},\hat{g}\}|_C,
\end{equation}
where the extension $\hat{f}$ (resp. $\hat{g}$) satisfies
\eqref{eq:ext}.

\medskip

\paragraph{\bf A word on coisotropic submanifolds (or first-class constraints)}
A submanifold $C$ of $M$ is called \emph{coisotropic} if
\begin{equation}\label{eq:coiso}
\pi^\sharp(TC^\circ)\subseteq TC.
\end{equation}
These submanifolds play a key role in Poisson geometry (see e.g.
\cite{CW}) and generalize the notion of first-class constraints in
physics \cite{Dir,Got2}: when $C$ is defined as the zero set of $k$
independent functions $\psi^i$, condition \eqref{eq:coiso} amounts
to $\{\psi^i,\psi^j\}|_C=0$.

Note that \eqref{eq:coiso} does not guarantee that the pullback
image of $\pi$ to $C$ is smooth; indeed, Example~\ref{ex:pullns}
illustrates a pullback image to a coisotropic submanifold that is
not smooth. Despite this fact, Dirac structures play an important
role in the study of coisotropic submanifolds: they provide a way to
extend the formulation of coisotropic embedding problems in
symplectic geometry (treated in \cite{Got1,Marle}) to the realm of
Poisson manifolds.

More specifically, assume that the pullback image of $L$ to $C$ is
smooth, so that it defines a Dirac structure
$\mathfrak{B}_\iota(L)$. Condition \eqref{eq:coiso} implies that
this Dirac structure has null distribution (c.f. \eqref{eq:bkker})
\begin{equation}\label{eq:coisoker}
TC\cap \mathfrak{B}_\iota(L) = \pi^\sharp(TC^\circ),
\end{equation}
which necessarily has (locally) constant rank (since the rank of the
left-hand side of \eqref{eq:coisoker}  is an upper semi-continuous
function on $C$, while the rank of the right-hand side is lower
semi-continuous, c.f. Section~\ref{subsec:facts}).

Using Dirac structures, one can then pose the converse question: can
any Dirac manifold with constant-rank null distribution be
coisotropically embedded into a Poisson manifold? It is proven in
\cite{CZ} that this is always possible, extending a result of Gotay
\cite{Got1} concerning coisotropic embeddings of constant-rank
presymplectic manifolds into symplectic manifolds. The reader can
find much more on coisotropic embeddings into Poisson manifolds in
\cite{CZ}.

\medskip

\paragraph{\bf Poisson-Dirac submanifolds and the Dirac bracket}

Let us assume that the pullback image \eqref{eq:bpi} is a smooth
bundle, hence defining a Dirac structure $\mathfrak{B}_\iota(L)$ on
$C$. As we have seen, this is guaranteed by the condition that
$\ker(\pi^\sharp)\cap TC^\circ$, or $\pi^\sharp(TC^\circ)$, has
constant rank. The Dirac structure $\mathfrak{B}_\iota(L)$ on $C$ is
defined by a {\it Poisson} structure if and only if its null
distribution is zero; i.e.,  by \eqref{eq:bkker}, if and only if
\begin{equation}\label{eq:poissondirac}
TC\cap \pi^\sharp(TC^\circ)= 0.
\end{equation}
In this case $C$ is called a \textit{Poisson-Dirac submanifold}
\cite{CF} (c.f. \cite[Prop. 1.4]{We83}). When $M$ is symplectic, a
Poisson-Dirac submanifold amounts to a symplectic submanifold (i.e.,
a submanifold whose natural presymplectic structure is
nondegenerate). In general, Poisson-Dirac submanifolds are
``leafwise symplectic submanifolds'', thus generalizing symplectic
submanifolds to the Poisson world.

On Poisson-Dirac submanifolds, all smooth functions are admissible.
Supposing that $\pi^\sharp(TC^\circ)$ has constant rank, the Poisson
bracket $\{\cdot,\cdot\}_C$ on $C$ may be computed as in
\eqref{eq:brks}:
\begin{equation}\label{eq:brkC}
\{f,g\}_C = \{\hat{f},\hat{g}\}|_C,
\end{equation}
where $\hat{f}$, $\hat{g}$ are local extensions of $f$, $g$ to $M$,
satisfying \eqref{eq:ext}.

Amongst Poisson-Dirac submanifolds, there are two special cases of
interest. The first one is when
$$
\pi^\sharp(TC^\circ)= 0,
$$
which is equivalent to $ TC^\circ \subseteq \ker(\pi^\sharp)$, or
$\image(\pi^\sharp)\subseteq TC$; this is precisely the condition
\eqref{eq:bf} for $\mathfrak{F}_\iota(\mathfrak{B}_\iota(L))=L$,
i.e., for the inclusion $\iota: C \hookrightarrow M$ to be a Poisson
map, and hence for $C$ to be a Poisson submanifold of $M$.

The second special case is when
$$
TM|_C = TC\oplus \pi^\sharp(TC^\circ),
$$
and $C$ is called a \textit{cosymplectic} submanifold. Such
submanifolds extend the notion of \textit{second-class constraints}
to the realm of Poisson manifolds; this is the context in which the
Dirac-bracket formula for the pullback Poisson structure
$\{\cdot,\cdot\}_C$ on $C$ can be derived, see e.g.
\cite[Chp.~10]{OR}, providing an alternative expression to
\eqref{eq:brkC} (see \eqref{eq:diracbrk} below).

When $C$ is a cosymplectic submanifold, one has natural projections
$$\mathrm{pr}_{TC}: TM|_C \to TC,\;\;\;
\mathrm{pr}_{\pi^\sharp(TC^\circ)} = \Id- \mathrm{pr}_{TC} :TM|_C
\to \pi^\sharp(TC^\circ),
$$
and it is clear from \eqref{eq:brkC} that the Poisson structure on
$C$ is given by the bivector field $ \pi_C = (\mathrm{pr}_{TC})_*
\pi|_C$. Note also that, for $f$, $g$ smooth functions on $C$ and
$F$ arbitrary local extensions of $f$ to $M$, we have (see
\eqref{eq:brhat}, recalling the extensions $\hat{f}$, $\hat{g}$ in
\eqref{eq:ext})
$$
dg(X_f) = \{f,g\}_C = -df(X_{\hat{g}}|_C) = -dF(X_{\hat{g}})|_C
=d\hat{g}(X_F)|_C
=dg(\mathrm{pr}_{TC}(X_F|_C));
$$
we conclude that the hamiltonian vector fields $X_f$ and $X_{F}$
(with respect to $\pi_C$ and $\pi$, respectively) are related by
\begin{equation}\label{eq:hamproj}
X_f = \mathrm{p}_{TC} (X_{F}|_C).
\end{equation}

Let us suppose that the ``constraint'' submanifold $C$ is defined by
$k$ independent functions $\psi^i$, i.e., $C=\Psi^{-1}(0)$ for a
submersion $\Psi=(\psi^1,\ldots,\psi^k): M\to \mathbb{R}^k$. Then
$d\psi^i$, $i=1,\ldots,k$, form a basis for $TC^\circ$, and $TC =
\ker(d\Psi)$ along $C$. Consider the matrix $(c^{ij})$, where
$$
c^{ij}=\{\psi^i,\psi^j\};
$$
the fact that $C$ is cosymplectic means that $\pi^\sharp(TC^\circ)$
has rank $k$ (since $TC$ has rank $\dim(M)-k$), so
$\pi^\sharp|_{TC^\circ} : TC^\circ \to \pi^\sharp(TC^\circ)$ is an
isomorphism, and $X_{\psi_i}=\pi^\sharp(d\psi_i)$, $i=1,\ldots,k$,
form a basis of $\pi^\sharp(TC^\circ)$. Hence, for
$a_j\in\mathbb{R}$, the condition $c^{ij}(x)a_j=0$ for all $i$ means
that $-d\psi^i(a_jX_{\psi^j})=0$ for all $i$, that is,
$a_jX_{\psi^j}\in TC\cap \pi^\sharp(TC^\circ)=\{0\}$, which implies
that $a_j=0$. So $C$ being cosymplectic means that the matrix
$(c^{ij})$ is invertible (in this case, one may also check that $k$
is necessarily even, so $C$ has even codimension).
Let us denote by $(c_{ij})$ its inverse. We have the following
explicit expression for the projection
$\mathrm{pr}_{\pi^\sharp(TC^\circ)}$: for $Y\in TM|_C$,
$$
\mathrm{pr}_{\pi^\sharp(TC^\circ)}(Y) = d\psi^i(Y)c_{ij}X_{\psi^j}.
$$
To verify this formula, note that if $Y\in TC$, then the right-hand
side above vanishes. On the other hand, for each $X_{\psi^i}$ in
$\pi^\sharp(TC^\circ)$,
$$
\mathrm{pr}_{\pi^\sharp(TC^\circ)}(X_{\psi^l}) =
d\psi^i(X_{\psi^l})c_{ij}X_{\psi^j} = c^{li}c_{ij}X_{\psi^j} =
\delta_j^l X_{\psi^j}= X_{\psi^l},
$$
so $\mathrm{pr}_{\pi^\sharp(TC^\circ)}$ is the identity on
$\pi^\sharp(TC^\circ)$. By \eqref{eq:hamproj}, we see that
$$
\{f,g\}_C = (\mathrm{pr}_{TC} (X_{F}|_C))g = (\mathrm{pr}_{TC}
(X_{F}|_C))G = ((\id - \mathrm{pr}_{\pi^\sharp(TC^\circ)}) X_{F}|_C)
G,
$$
where $F$, $G$ are arbitrary local extensions of $f$, $g$. So
\begin{equation}\label{eq:diracbrk}
\{f,g\}_C = (\{F,G\} - \{F,\psi^i\}c_{ij}\{\psi^j,G\})|_C,
\end{equation}
which is Dirac's bracket formula for cosymplectic manifolds.

\medskip

\paragraph{\bf Momentum level sets}

Another important class of submanifolds of Poisson manifolds is
given by level sets of conserved quantities associated with
symmetries, such as momentum maps.

Suppose that a Lie group $G$ acts on a Poisson manifold $M$ in a
hamiltonian fashion, with equivariant momentum map $J: M\to \gstar$.
This means that
\begin{equation}\label{eq:momcond}
u_M = X_{\SP{J,u}},
\end{equation}
where $u_M$ is the infinitesimal generator of the action associated
with $u\in \frakg$, and the equivariance of $J$ is with respect to
the coadjoint $G$-action on $\gstar$; infinitesimally, the
equivariance condition implies that $J$ is a Poisson map with
respect to the Lie-Poisson structure on $\gstar$ (see e.g.
\cite[Sec.~10]{OR} for details). We denote by $D\subseteq TM$ the
distribution on $M$ tangent to the $G$-orbits.

If $\mu\in \gstar$ is a regular value for $J$, one considers the
submanifold $C = J^{-1}(\mu)$, which naturally carries an action of
$G_\mu$, the isotropy group of the coadjoint action at $\mu$. Let us
suppose that this $G_\mu$-action on $J^{-1}(\mu)$ is free and
proper, so that its orbit space is a smooth manifold,  and we can
consider the following diagram
\begin{equation}\label{eq;diagram}
 \xymatrix {
J^{-1}(\mu)  \ar[r]^\iota \ar[d]_p  & M \\
J^{-1}(\mu)/G_\mu,
 }
\end{equation}
where $\iota$ is the inclusion and $p$ is the quotient map, which is
a surjective submersion. The reduction theorem in this context (see
e.g. \cite[Sec.~10.4.15]{OR}) says that $J^{-1}(\mu)/G_\mu$ inherits
a natural Poisson structure $\pi_\mu$, with corresponding bracket
$\{\cdot,\cdot\}_\mu$, characterized by the following condition:
\begin{equation}\label{eq:brkcond}
\{f,g\}_\mu \circ p = \{\bar{f},\bar{g}\}\circ \iota,
\end{equation}
where $\bar{f}$, $\bar{g}$ are local extensions of $p^*f$, $p^*g \in
C^\infty(J^{-1}(\mu))$ to $M$ satisfying $d\bar{f}|_{D}= 0$,
$d\bar{g}|_D=0$. In this picture, both $M$ and $J^{-1}(\mu)/G_\mu$
are Poisson manifolds, but there is no geometric interpretation of
the ``intermediate'' manifold $C = J^{-1}(\mu)$.

When the Poisson structure on $M$ is nondegenerate, i.e., defined by
a symplectic form $\omega$, we are in the classical setting of
Marsden-Weinstein reduction. In this case $J^{-1}(\mu)$ is naturally
a presymplectic manifold, with 2-form $\iota^*\omega$, and the
Poisson structure $\pi_\mu$ on $J^{-1}(\mu)/G_\mu$ is defined by a
symplectic form $\omega_\mu$; the condition relating $\pi$ and
$\pi_\mu$ is easily expressed in term of the corresponding
symplectic forms as
\begin{equation}\label{eq:formcond}
p^*\omega_\mu = \iota^* \omega,
\end{equation}
which determines $\omega_\mu$ uniquely.
As we now see,  Dirac structures allow us to identify the geometric
nature of $J^{-1}(\mu)$ in general and interpret \eqref{eq:brkcond}
as a direct generalization of \eqref{eq:formcond}.

We start by considering the backward image of $\pi$ with respect to
$\iota$. The clean-intersection condition guaranteeing the
smoothness of $\mathfrak{B}_\iota(L_\pi)$ is that the distribution
$$
\ker(dJ)^\circ \cap L_\pi = \image((dJ)^*)\cap L_\pi =
\image((dJ)^*)\cap\ker(\pi^\sharp)
$$
has constant rank over $J^{-1}(\mu)$. Since $\mu$ is a regular value
of $J$ and, for $x\in J^{-1}(\mu)$,
$$
\image((dJ)^*)\cap\ker(\pi^\sharp)|_x=\{\alpha \in T^*_xM\,|\,
\alpha=(dJ)^*(u), \, \pi^\sharp(\alpha)= u_M(x) =0 \}=
(dJ)^*(\mathfrak{g}_x),
$$
where $\frakg_x$ is the isotropy Lie algebra of the $G$-action at
$x$, we see that the clean-intersection condition amounts to
checking that $\dim(\frakg_x)$ is constant for $x\in J^{-1}(\mu)$.
Since $J$ is equivariant, $G_x \subseteq G_\mu$, and $G_x$ agrees
with the isotropy group of $G_\mu$ at $x$, which is trivial since
$G_\mu$ acts on $J^{-1}(\mu)$ freely. Hence $\dim(\frakg_x)=0$ for
$x\in J^{-1}(\mu)$.
It follows that $\mathfrak{B}_\iota(L_\pi)$ makes the momentum level
set $J^{-1}(\mu)$ into a Dirac manifold in such a way that $\iota$
is a b-Dirac map.

Since $\pi^\sharp(\ker(dJ)^\circ)=D$ (by \eqref{eq:momcond}), the
null distribution of $\mathfrak{B}_\iota(L_\pi)$ is (c.f.
\eqref{eq:bkker})
$$
T(J^{-1}(\mu))\cap D = \ker(p),
$$
where the equality follows from the equivariance of $J$ and the fact
that $\ker(p)$ agrees with the distribution tangent to the
$G_\mu$-orbits on $J^{-1}(\mu)$. By Example~\ref{ex:push} and
Remark~\ref{rem:group}, the forward image of
$\mathfrak{B}_\iota(L_\pi)$ under $p$ defines a Poisson structure on
$J^{-1}(\mu)/G_\mu$, with bracket given by
\begin{equation}\label{eq:redp}
\{f,g\}_\mu \circ p = \{p^*f, p^*g\}_C,
\end{equation}
where $\{\cdot,\cdot\}_C$ is the Poisson bracket on admissible
functions on $J^{-1}(\mu)$. By \eqref{eq:brks},
$$
\{p^*f, p^*g\}_C = \{\widehat{p^*f},\widehat{p^*g}\}|_C,
$$
where $\widehat{p^*f}$ denotes a local extension of $p^*f$ to $M$
with differential vanishing on $\pi^\sharp(TC^\circ)=
\pi^\sharp(\image(dJ^*))=D$. So the bracket $\{\cdot,\cdot\}_\mu$ in
\eqref{eq:redp} is the same as the bracket defined in
\eqref{eq:brkcond}.

In conclusion, Dirac geometry completes the reduction diagram
\eqref{eq;diagram} in the following way: the level set $J^{-1}(\mu)$
inherits a Dirac structure $\mathfrak{B}_\iota(L_\pi)$ for which
$\iota$ is a b-Dirac map, while the Poisson structure $\pi_\mu$ on
the quotient $J^{-1}(\mu)/G_\mu$ is its forward image, so $p$ is an
f-Dirac map. As mentioned in Example~\ref{ex:push}, $p$ is also a
b-Dirac map, and this condition uniquely determines $\pi_\mu$. We
can then write the relationship between $\pi$ and $\pi_\mu$ (c.f.
\eqref{eq:brkcond}) exactly like we do for symplectic forms:
$$
\mathfrak{B}_\iota(L_\pi) = \mathfrak{B}_p(L_{\pi_\mu}),
$$
noticing that, in the symplectic setting,
$\mathfrak{B}_\iota(L_\omega) = \mathfrak{B}_p(L_{\omega_\mu})$ is
nothing but \eqref{eq:formcond}.

\section{Brief remarks on further developments}\label{sec:survey}

In this last section, we merely indicate some directions of more
recent developments and applications of Dirac structures. Although
these lectures have mostly focused on how Dirac structures arise in
the study of constraints of classical systems on Poisson manifolds,
subsequent work to Courant's paper \cite{courant} revealed
connections between Dirac structures and many other subjects in
mathematics and mathematical physics. Even in mechanics, (almost)
Dirac structures have found applications beyond their original
motivation, such as the study of implicit hamiltonian systems
(including nonholonomic systems, see \cite{schaft} for a survey and
original references) as well as implicit lagrangian systems and
variational principles, see e.g. \cite{MY}.

Understanding the properties of the Courant bracket \eqref{eq:cour},
used to formulate the integrability condition for Dirac structures,
has led to the notion of \textit{Courant algebroid} \cite{LWX},
which consists of a vector bundle $E\to M$, a symmetric
nondegenerate fibrewise pairing on $E$, a bracket on the space
$\Gamma(E)$, and a bundle map $E\to TM$ satisfying axioms that
emulate those of $\TM\to M$ equipped with $\SP{\cdot,\cdot}$,
$\Cour{\cdot,\cdot}$ and $\pr_T: \TM\to TM$, see
Section~\ref{sec:dirac} (a detailed discussion can be found in
\cite{kos}). Dirac structures on general Courant algebroids are used
in \cite{LWX} to extend the theory of Lie bialgebras, Manin triples
and Poisson-Lie groups to the realm of Lie algebroids and groupoids.
From another perspective, Courant algebroids admit a natural
super-geometric description in terms of certain types of
differential graded symplectic manifolds \cite{Roy2,severa}; in this
context, Dirac structures may be appropriately viewed as analogs of
lagrangian submanifolds. Courant algebroids have also been studied
in connection with field theories and vertex algebras (see e.g.
\cite{AS,Br} and references therein).

A distinguished class of Courant algebroids, introduced by Severa
(see e.g. \cite{SW}), is defined through a modification of the
Courant bracket \eqref{eq:cour} on $\TM$ by a closed 3-form $H \in
\Omega^3(M)$:
\begin{equation}\label{eq:Hcour}
\Cour{(X,\alpha),(Y,\beta)}_H = \Cour{(X,\alpha),(Y,\beta)}+
i_Yi_XH.
\end{equation}
Such Courant algebroids are known as \textit{exact}, and Dirac
structures relative to $\Cour{\cdot,\cdot}_H$ are often referred to
as \textit{$H$-twisted}. For example, the integrability condition
for a bivector field $\pi$ with respect to \eqref{eq:Hcour} becomes
$$
\Jac_\pi(f,g,h)=H(X_f,X_g,X_h),
$$
see \cite{KS}. A key example of an $H$-twisted Dirac structure is
the so-called \textit{Cartan-Dirac} structure (see e.g.
\cite{ABM,bcwz}), defined on any Lie group $G$ whose Lie algebra is
equipped with a non-degenerate Ad-invariant quadratic form; in this
case, $H \in \Omega^3(G)$ is the associated Cartan 3-form.

As mentioned in Section~\ref{sec:properties}, any Dirac structure
$L\subset \TM$ carries a Lie algebroid structure. From a
Lie-theoretic standpoint, it is natural to search for the class of
Lie groupoids ``integrating'' Dirac structures. For Poisson
structures, the corresponding global objects are the so-called
\textit{symplectic groupoids} \cite{CDW} (see also \cite{catfel},
\cite{CF2}), originally introduced as part of a quantization scheme
for Poisson manifolds (see e.g. \cite{BW}). More generally, the
global counterparts of Dirac structures were identified in
\cite{bcwz}, and called \textit{presymplectic groupoids}. Passing
from Dirac structures to presymplectic groupoids is key to uncover
links between Dirac structures, momentum maps and equivariant
cohomology, see e.g. \cite{bcwz,xu}.

The connection between Dirac structures and the theory of momentum
maps is far-reaching. As recalled in Section~\ref{sec:back},
classical momentum maps in Poisson geometry are defined by Poisson
maps
$$
J: M\to \gstar
$$
from a Poisson manifold into the dual of a Lie algebra. In the last
25 years, generalized notions of momentum maps have sprung up in
symplectic geometry, mostly motivated by the role of symplectic
structures in different areas of mathematics and mathematical
physics; see \cite{We} for a general discussion and original
references. Prominent examples include the theory of hamiltonian
actions of symplectic groupoids \cite{MiWe}, in which arbitrary
Poisson maps may be viewed as momentum maps, and the theory of
$G$-valued moment maps, which arises in the study of moduli spaces
in gauge theory \cite{AMM}. These seemingly unrelated new notions of
momentum map turn out to fit into a common geometric framework
provided by Dirac structures and Courant algebroids. Much of the
usual hamiltonian theory in Poisson geometry carries over to the
realm of Dirac structures, and the key observation is that, just as
classical momentum maps are Poisson maps, many generalized momentum
maps are morphisms between Dirac manifolds (in an appropriate sense,
see Remark~\ref{rem:transv}); see e.g. \cite{ABM,BC,bcwz,BIS}.

Another viewpoint to Dirac structures that has proven fruitful is
based on their description via  \textit{pure spinors} \cite{AX,Hit}
(see also \cite{ABM}). The bundle $\TM$, equipped with the symmetric
pairing \eqref{eq:pair}, gives rise to a bundle $\Cl(\TM)$ of
Clifford algebras, and the exterior algebra $\wedge^\bullet T^*M$
carries a natural representation of $\Cl(\TM)$, defined by the
action
$$
(X,\alpha)\cdot \phi = i_X\phi +
\alpha\wedge \phi,
$$
for $(X,\alpha)\in \TM$ and $\phi \in \wedge^\bullet T^*M$. In this
way, $\wedge^\bullet T^*M$ is viewed as a spinor module over
$\Cl(\TM)$. The subspace $L\subset \TM$ that annihilates a given
(nonzero) $\phi$ is always isotropic, and $\phi$ is called a
\textit{pure spinor} when $L$ is lagrangian: $L=L^\perp$. Pure
spinors corresponding to the same annihilator $L$ must agree up to
rescaling; as a result, almost Dirac structures on $M$ are uniquely
characterized by a (generally nontrivial) smooth line bundle
\begin{equation}\label{eq:line}
l \subset \wedge^\bullet T^*M,
\end{equation}
generated by pure spinors at all points. The integrability of $L$
relative to the ($H$-twisted) Courant bracket  can be naturally
expressed in terms of $l$ as well, see e.g \cite{AX,Gua}. In this
way, Dirac structures are represented  by specific line bundles
\eqref{eq:line} and, by taking (local) sections, one obtains
concrete descriptions of Dirac structures through differential
forms. For example, the Dirac structure $L_\omega$ \eqref{eq:Lomega}
is the annihilator of the differential form $e^{-\omega} \in
\Omega^\bullet(M)$. This approach to Dirac structures leads to
natural constructions of invariant volume forms associated with
$G$-valued moment maps \cite{ABM} and has been an essential
ingredient in the study of generalized complex structures
\cite{Hit,Gua}, that we briefly recall next.

{\it Generalized complex structures} \cite{Gua,Hit} are particular
types of complex Dirac structures $L\subset \TM\otimes \mathbb{C}$,
with the extra property that $L\cap \overline{L}=\{0\}$; they
correspond to $+i$-eigenbundles of endomorphisms $\mathcal{J}: \TM
\to \TM$ satisfying $\mathcal{J}^2 = -1$ and preserving the pairing
\eqref{eq:pair}. These geometrical structures unify symplectic and
complex geometries, though the most interesting examples fall in
between these extreme cases (see e.g. \cite{CG}). Given a complex
structure $J:TM\to TM$ and a symplectic structure $\omega\in
\Omega^2(M)$, viewed as a bundle map $TM\to T^*M$, the associated
generalized complex structures $\TM \to \TM$ are
$$
\begin{pmatrix}
J & 0\\
0 & -J^*
\end{pmatrix}, \qquad
\begin{pmatrix}
0 & -\omega^{-1}\\
\omega & 0
\end{pmatrix}.
$$
Besides their mathematical interest, generalized complex structures
have drawn much attention due to their strong ties with theoretical
physics, including e.g. topological strings and super symmetric
sigma models, see e.g. \cite{kapli,zabzine} and references therein.


\end{document}